\documentclass[12pt]{article}
\usepackage{amsmath}
\usepackage{amssymb}
\usepackage{amsthm}
\usepackage{mathrsfs}
\usepackage{appendix}
\usepackage{hyperref}
\usepackage{geometry}
\usepackage{authblk}
\usepackage{tikz}
\usepackage{titlesec}
\usepackage{stmaryrd}

\usetikzlibrary{positioning, shapes.geometric}

\newtheorem{definition}{Definition}
\newtheorem{lemma}{Lemma}
\newtheorem{theorem}{Theorem}
\newtheorem{proposition}{Proposition}
\newtheorem{remark}{Remark}[section]

\newtheorem{assumption}{Assumption}

\newcommand{\e}{\mathrm{e}}
\renewcommand{\i}{\mathrm{i}}
\renewcommand{\d}{\mathrm{d}}

\title{Almost global existence for Ultra-differential Hamiltonian PDEs in  $L^2$ space}
\date{}
\author[1]{Bingqi Yu}
\author[,1,2]{Yong Li\thanks{Corresponding author}}
\affil[1]{School of Mathematics, Jilin University, Changchun 130012, People’s Republic of China.}
\affil[2]{Center for Mathematics and Interdisciplinary Sciences,
		Northeast Normal University, Changchun 130024, People’s Republic of China.}
\allowdisplaybreaks

\begin{document}
	
	\maketitle
		\footnote{E-mail address: \url{yubq23@mails.jlu.edu.cn}(B. Yu),  \url{liyong@jlu.edu.cn}(Y. Li)}
	
	\begin{abstract}
		This paper combines the decay of high modes with the smallness introduced by high orders, leading to a normal form lemma for infinite-dimensional Hamiltonian systems under ultra-differentiable regularity. We prove the sub-exponential stability time of a wide class of Hamiltonian PDEs, including the Schr\"odinger equation with convolution potentials, fractional-order Schr\"odinger equations, and beam equations with metrics. When the conditions are equivalent to previous ones, the stability time we obtain reaches Bourgain's predicted optimal bound. Furthermore, we approach earlier results under lower conditions. These results are discussed within a general framework we propose, which applies to the ultra-differential class.
		
	\end{abstract}
	
	\section{Introduction}

	\subsection{Effective stability for Hamiltonian systems}
	
	The long-time stability of Hamiltonian systems under small 
	perturbations is a central problem in dynamical systems and 
	mathematical physics. A fundamental question is whether, and 
	for how long, the qualitative features of a solution, such 
	as the size of the initial data, can be preserved under a 
	small perturbation of order $\varepsilon$. This question is 
	sometimes referred to as \emph{effective stability}, since 
	one seeks not permanent stability but rather stability up to 
	a time that is large relative to $1/\varepsilon$.
	
	In the finite-dimensional setting, this question is answered 
	by the Nekhoroshev theorem \cite{N77,P93,B99,BG24}: for analytic nearly 
	integrable Hamiltonian systems, the action variables remain 
	stable up to time of order $\e^{1/\varepsilon}$. This 
	exponential stability time reflects the analytic regularity 
	in an essential way; in the finitely differentiable case, the 
	stability time degrades to a polynomial in $1/\varepsilon$ 
	\cite{B10}.
	
	In the infinite-dimensional setting, the analogous problem 
	concerns Hamiltonian PDEs with small initial data. Here the 
	question takes the form of \emph{almost global existence}: 
	one seeks solutions that remain small -- in a suitable 
	function space norm -- up to a time that is as long as 
	possible in terms of the initial size $\varepsilon$. This 
	problem has attracted significant attention over the past 
	three decades, with results obtained for a wide range of 
	equations including nonlinear Schr\"odinger, wave, and beam 
	equations on compact manifolds 
	\cite{BDGS07,BFM24,BFG20,BMP19,BMP20,B00,FGL13,FM24,W15}. 
	 Under 
	finite-order differentiable conditions, polynomial-in-$(1/\varepsilon)$ 
	stability time is achievable \cite{BDGS07,BFM24,BFG20,B00,W15}. 
	When the regularity is promoted to the analytic or Gevrey 
	class, one expects a dramatic improvement, but the precise 
	optimal bound has long been an open question.
	
	A landmark result in this direction is due to Benettin, 
	Fr\"ohlich and Giorgilli \cite{B88}, who established 
	stability time of order $|\ln\varepsilon|^2/\ln|\ln\varepsilon|$ 
	for systems with finite-range coupling. In the general 
	analytic case, Bourgain \cite{B04} predicted that this same 
	quantity
	$$
	T \sim \exp\!\left(\frac{|\ln\varepsilon|^2}{\ln|\ln\varepsilon|}\right)
	$$
	should represent the optimal stability time for 
	infinite-dimensional Hamiltonian systems. Subsequent works 
	confirmed stability times of the weaker form 
	$|\ln\varepsilon|^{1+a}$ under analytic or Gevrey-like 
	conditions \cite{BMP20,CMW20,FG13,P18}, and Bourgain's 
	predicted bound was recently realized for the 
	Schr\"odinger--Poisson equation \cite{BCGW24}. The question 
	of achieving this optimal bound within a \emph{general} 
	framework, and of extending the analysis to classes of 
	regularity beyond Gevrey, remains open. The present 
	paper addresses both of these problems.

	\subsection{Main results}
	Before stating our results, we recall that the 
	\emph{ultra-differentiable} function classes form a large 
	family of regularity spaces lying strictly between $C^\infty$ 
	and the analytic class. They are defined by prescribing an 
	explicit decay rate for the Fourier coefficients, of the form 
	$\exp(-sf(|j|))$ for a suitable weight function $f$. The 
	Gevrey class (corresponding to $f(x)=x^\theta$, $0<\theta<1$) 
	and the logarithmic ultra-differentiable class (corresponding 
	to $f(x)=(\ln(|j|+\kappa))^q$, $q>1$) are two canonical 
	representatives of this family, but many other choices of 
	$f$ are possible. The results of this paper are stated for 
	these two representative cases; however, as will be clear 
	from the proof, the method applies to any weight function 
	$f$ satisfying the mild sub-additivity condition in 
	Assumption~1, and yields a stability time that can be 
	computed explicitly from $f$ via the relation \eqref{balance} 
	below.
	
	We mainly consider the Hamiltonian system 
	\[H=H_0+P,H_0=\sum_{j\in\mathbb{Z}^\mathsf{d}}\omega_j|u_j|^2,\]
	where $P$ is 
	a perturbation as in \eqref{Hamiltonian}. The precise definitions of 
	the weighted spaces $W_{s,\theta}^G$ (Gevrey class) and $W_{s,q}^U$ 
	(logarithmic ultra-differentiable class), the norm $\|\cdot\|_s$, and 
	Assumption~2 on the frequencies are given in Sections~\ref{sec:2} 
	and~\ref{sec:3}. Our two main abstract results are as follows.
	
	\begin{theorem}[Gevrey class]\label{res1}
		For Hamiltonian \eqref{Hamiltonian} with initial datum $u(0)=u_0$ 
		in the $\theta$-Gevrey space $W_{s,\theta}^G$, assume the frequencies 
		$\omega_j$ satisfy \textbf{Assumption~2} with $\beta>1$. Then for 
		sufficiently large $s$, there exist a threshold $\varepsilon_0>0$ and 
		constants $C_{\mathtt{sta}},C_{\mathtt{fin}}>0$ such that the following 
		holds: if $u(0)$ is real and
		$$\varepsilon:=\|u(0)\|_{s}<\varepsilon_0,$$
		then
		$$\sup_{|t|\leq T_{\varepsilon}}\|u(t)\|_{s}<C_{\mathtt{sta}}\varepsilon,$$
		where
		$$T_{\varepsilon}>\frac{1}{C_{\mathtt{sta}}}
		\exp\!\left(C_{\mathtt{fin}}
		\frac{|\ln\varepsilon|^2}{\ln|\ln\varepsilon|}\right).$$
	\end{theorem}
	
	\begin{remark}
		This result achieves the stability time conjectured by Bourgain 
		\cite{B04} for infinite-dimensional Hamiltonian systems, now 
		established in the weighted $L^2$ space and within a general 
		framework. The same bound was recently obtained in the weighted 
		$L^1$ space for the Schr\"odinger--Poisson equation \cite{BCGW24}. 
		Our result improves upon the earlier $|\ln\varepsilon|^{1+a}$-type 
		bounds of \cite{BMP20,CMW20,FG13,P18}.
	\end{remark}
	
	\begin{theorem}[Logarithmic ultra-differentiable class]\label{res2}
		For Hamiltonian \eqref{Hamiltonian} with initial datum $u(0)=u_0$ 
		in the ultra-differentiable space $W_{s,q}^U$, assume the frequencies 
		$\omega_j$ satisfy \textbf{Assumption~2} with $\beta>1$. Then for 
		sufficiently large $s$, there exist a threshold $\varepsilon_0>0$ and 
		constants $C_{\mathtt{sta}},C_{\mathtt{fin}}>0$ such that the following 
		holds: if $u(0)$ is real and
		$$\varepsilon:=\|u(0)\|_{s}<\varepsilon_0,$$
		then
		$$\sup_{|t|\leq T_{\varepsilon}}\|u(t)\|_{s}<C_{\mathtt{sta}}\varepsilon,$$
		where
		$$T_{\varepsilon}>\frac{1}{C_{\mathtt{sta}}}
		\exp\!\left(C_{\mathtt{fin}}|\ln\varepsilon|^{1+a}\right)$$
		for any $a<\dfrac{q-1}{qp+1}$.
	\end{theorem}
	
	\begin{remark}
		This theorem identifies the weakest regularity sufficient to 
		achieve $\exp(|\ln\varepsilon|^{1+a})$-type stability time. 
		Previous results under the same ultra-differentiable conditions 
		obtained only $\exp(|\ln\varepsilon|/\ln|\ln\varepsilon|)$ for 
		$1<q<2$ \cite{FM23} and $\exp(|\ln\varepsilon|\ln|\ln\varepsilon|)$ 
		for $q=2$ \cite{CCMW22}. For the specific case $q=2$, our bound 
		$|\ln\varepsilon|^{4/3}$ improves upon $|\ln\varepsilon|^{5/4}$ 
		of \cite{CLW24} and $|\ln\varepsilon|^{14/13}$ of \cite{SLW23}.
	\end{remark}
	
	The generality of Assumptions~1--2 allows the above theorems to be 
	applied to a broad class of Hamiltonian PDEs. We illustrate this 
	with three representative equations; the verification of the 
	assumptions and complete proofs are given in Section~\ref{sec:app}.
	
	\subsubsection*{Schr\"{o}dinger equation with convolution potential}
	
	We consider
	\begin{equation}\label{ex1}
		\mathrm{i}\partial_t\psi = -\Delta\psi + V*\psi + p(|\psi|^2)\psi,
		\quad x\in\mathbb{T}^\mathsf{d},
	\end{equation}
	where $V$ is a convolution potential and $p\in C^\infty(\mathbb{R},\mathbb{R})$ 
	with $p(0)=0$. The frequencies are $\omega_j = |j|^2 + V_j$, which 
	satisfy Assumption~2 with $\beta=2$ and $p=1$, for a full-measure 
	set of potentials $V$.
	
	\begin{theorem}[See Theorems \ref{ex1thm1},\ref{ex1thm2}  below]\label{thm:NLS}
		There exists a zero-measure set $\mathcal{V}^{\mathtt{res}}\subset\mathcal{V}$ 
		such that for all $V\in\mathcal{V}\setminus\mathcal{V}^{\mathtt{res}}$, 
		$s>S_{\mathtt{fin}}$, and $\varepsilon<\varepsilon_0$, if 
		$\|\psi_0\|_{s,\theta}^G=\varepsilon$, then the solution of 
		\eqref{ex1} satisfies
		$$\|\psi(t)\|_{s,\theta}^G \leq C_{\mathtt{sta}}\varepsilon, 
		\quad \forall\, |t| \leq \frac{1}{C_{\mathtt{sta}}}
		\exp\!\left(C_{\mathtt{fin}}
		\frac{|\ln\varepsilon|^2}{\ln|\ln\varepsilon|}\right).$$
		The same holds in the ultra-differentiable norm $\|\cdot\|_{s,q}^U$ 
		with stability time 
		$\exp(C_{\mathtt{fin}}|\ln\varepsilon|^{1+a})$, 
		$a\leq\frac{q-1}{q+1}$.
	\end{theorem}
	\begin{remark}
		This equation illustrates that the classical Bourgain non-resonance 
		condition for convolution potentials \cite{B05}, as extended in 
		\cite{BMP20}, is a special case of our Assumption~2 with $p=1$.
	\end{remark}

	\subsubsection*{Fractional Schr\"{o}dinger equation}
	
	We consider
	\begin{equation}\label{ex2}
		\mathrm{i}\partial_t\psi = (\Delta+m)^\eta\psi + p(|\psi|^2)\psi,
	\end{equation}
	where $\eta>\frac{1}{2}$. The frequencies are 
	$\omega_j=(|j|^2+m)^\eta$, satisfying Assumption~2 with 
	$\beta=2\eta$ and $p=2$, for almost every mass parameter $m$.
	
	\begin{theorem}[See Theorems \ref{ex2thm1},\ref{ex2thm2}  below]\label{thm:fNLS}
		For any interval $[M_1,M_2]$, there exists a zero-measure set 
		$\mathcal{M}\subset[M_1,M_2]$ such that for all 
		$m\in[M_1,M_2]\setminus\mathcal{M}$, $s>S_{\mathtt{fin}}$, 
		and $\varepsilon<\varepsilon_0$, if 
		$\|\psi_0\|_{s,\theta}^G=\varepsilon$, then the solution of 
		\eqref{ex2} satisfies
		$$\|\psi(t)\|_{s,\theta}^G \leq C_{\mathtt{sta}}\varepsilon, 
		\quad \forall\, |t| \leq \frac{1}{C_{\mathtt{sta}}}
		\exp\!\left(C_{\mathtt{fin}}
		\frac{|\ln\varepsilon|^2}{\ln|\ln\varepsilon|}\right).$$
		The same holds in the ultra-differentiable norm $\|\cdot\|_{s,q}^U$ 
		with stability time 
		$\exp(C_{\mathtt{fin}}|\ln\varepsilon|^{1+a})$, 
		$a\leq\frac{q-1}{3q+1}$.
	\end{theorem}
	\begin{remark}
		This equation illustrates a case where the non-resonance condition 
		requires $p=2$, arising from the use of a single scalar parameter 
		$m$ to control resonances among frequencies with stronger separation.
	\end{remark}

	\subsubsection*{Beam equation with metric}
	
	We consider
	\begin{equation}\label{ex3}
		\psi_{tt}+\Delta_g^2\psi+m\psi 
		= -\frac{\partial p}{\partial\psi}
		+\sum_{l=1}^{L}\partial_{x_l}
		\frac{\partial p}{\partial(\partial_l\psi)},
	\end{equation}
	on the torus $\mathbb{T}^\mathsf{d}$ equipped with a Riemannian metric 
	$g$. The frequencies are $\omega_j=\sqrt{|j|_g^4+m}$, satisfying 
	Assumption~2 with $\beta=2$ and $p=2$, for a full-measure set of 
	metrics $g$.
	
	\begin{theorem}[See Theorems \ref{ex3thm1},\ref{ex3thm2} below]\label{thm:beam}
		For $0<\zeta_1<\zeta_2$, there exists a zero-measure set 
		$\mathcal{G}^{\mathtt{res}}\subset\mathcal{G}_0(\zeta_1,\zeta_2)$ 
		such that for all 
		$g\in\mathcal{G}_0(\zeta_1,\zeta_2)\setminus\mathcal{G}^{\mathtt{res}}$, 
		$s>S_{\mathtt{fin}}$, and $\varepsilon<\varepsilon_0$, if 
		$\|\psi_0\|_{s,\theta}^G=\varepsilon$, then the solution of 
		\eqref{ex3} satisfies
		$$\|\psi(t)\|_{s,\theta}^G \leq C_{\mathtt{sta}}\varepsilon, 
		\quad\forall\,|t| \leq \frac{1}{C_{\mathtt{sta}}}
		\exp\!\left(C_{\mathtt{fin}}
		\frac{|\ln\varepsilon|^2}{\ln|\ln\varepsilon|}\right).$$
		The same holds in the ultra-differentiable norm $\|\cdot\|_{s,q}^U$ 
		with stability time 
		$\exp(C_{\mathtt{fin}}|\ln\varepsilon|^{1+a})$, 
		$a\leq\frac{q-1}{3q+1}$.
	\end{theorem}
	\begin{remark}
		This equation illustrates that the metric $g$ can serve as a 
		non-resonance parameter in place of a scalar, with the algebraic 
		structure of $\mathcal{G}$ playing the role that Diophantine 
		conditions play in the convolution case.
	\end{remark}

	\subsection{Main techniques and contributions}
	
	We summarize the main contributions of this paper in two aspects:
	the novelty of the results and the new techniques developed.
	
	\medskip
	\noindent\textbf{Novelty of the results.}
	The first main contribution is to establish Bourgain's conjectured
	optimal stability time within a \emph{general} abstract framework,
	rather than for a single specific equation. Previous works achieving
	the $|\ln\varepsilon|^2/\ln|\ln\varepsilon|$ bound, such as
	\cite{BCGW24}, were tailored to specific equations and specific
	function spaces. Our Theorem~\ref{res1} obtains the same bound in
	the weighted $L^2$ space for a broad class of Hamiltonian PDEs whose
	frequencies satisfy the separation properties in Assumption~2.
	
	The second contribution concerns the ultra-differentiable regularity
	class. Our Theorem~\ref{res2} establishes stability time of order
	$\exp(|\ln\varepsilon|^{1+a})$ with $a < (q-1)/(qp+1)$, which
	significantly improves upon all previously known bounds in this
	class. A comparison for the representative case $q=2$ is
	illustrative:
	\begin{center}
		\begin{tabular}{lll}
			\hline
			Reference & Stability time & Condition \\
			\hline
			\cite{CLW24} & $|\ln\varepsilon|^{5/4}$ & $q=2$ \\
			\cite{SLW23} & $|\ln\varepsilon|^{14/13}$ & $q=2$ \\
			\cite{CCMW22} & $|\ln\varepsilon|\ln|\ln\varepsilon|$ & $q=2$ \\
			This paper   & $|\ln\varepsilon|^{4/3}$ & $q=2,\,p=1$ \\
			\hline
		\end{tabular}
	\end{center}
	
	The third contribution is a generalization of the non-resonance
	condition. Assumption~2 introduces an exponent $p\geq1$ that
	controls the strength of the small-divisor estimate. The classical
	Bourgain condition for convolution potentials \cite{B05,BMP20}
	corresponds to $p=1$, while weaker non-resonance conditions arising
	from a single scalar parameter, as in the fractional Schr\"odinger
	and beam equations, correspond to $p=2$. Our framework thus
	unifies and extends several non-resonance conditions appearing
	separately in the literature.
	
	\medskip
	\noindent\textbf{Main techniques.}
	The central difficulty in establishing long-time stability for
	infinite-dimensional Hamiltonian systems is to simultaneously
	control the growth of high Fourier modes and the accumulation of
	errors during the iterative normal form procedure. The main technical innovation of this
	paper is to transplant the truncation estimate approach of the
	finite-dimensional Nekhoroshev theory into the infinite-dimensional
	ultra-differentiable setting.
	
	Concretely, we decompose the phase space into low modes ($|j|\leq N$)
	and high modes ($|j|>N$), and perform a Birkhoff-type normal form
	iteration on the low-mode part. A key difficulty absent in the
	finitely differentiable case of \cite{BFM24} is that the number of
	iteration steps is unbounded, so the dependence of all estimates on
	the step count must be tracked explicitly throughout the iteration.
	After $d$ steps, two remainder terms of different origins remain:
	a high-degree remainder controlled by $r^d$, and a high-mode
	remainder controlled by $\e^{-(s-s_0)f(N)}$. The Normal Form Lemma
	is then obtained by choosing the truncation parameter $N$ and the
	iteration depth $d$ so that these two remainders are of the same
	order, via the implicit relation
	\begin{equation}\label{balance}
		d^p \ln(dN) = \frac{f(N)}{d}.
	\end{equation}
	Solving this relation for two representative weight functions
	$f(x)=x^\theta$ and $f(x)=(\ln x)^q$ -- using the Lambert
	$W$-function -- yields the stability times in
	Theorems~\ref{res1} and~\ref{res2} respectively. The high-mode
	dynamics are then controlled separately using the block-diagonal
	structure of the normal form and a bootstrap argument.
	
	\medskip
	\noindent\textbf{Organization of the paper.}
	The remaining sections of this paper are organized as follows.
	Section~\ref{sec:2} introduces the functional setting, including
	the abstract weighted space $W_s$, the Gevrey and ultra-differentiable
	subclasses, and the class of perturbations under consideration.
	Section~\ref{sec:3} formulates the resonance structure and the
	non-resonance Assumption~2, and establishes the solution to the
	homological equation.
	Section~\ref{sec:4} carries out the iterative normal form
	construction with explicit coefficient estimates at each step.
	Section~\ref{sec:5} combines the iteration with the truncation
	estimates to prove the Normal Form Lemma, and computes the resulting
	stability time for each regularity class.
	Section~\ref{sec:6} completes the proof of the main stability
	theorems via a decomposition into low- and high-mode dynamics
	and a bootstrap argument.
	Section~\ref{sec:app} applies the abstract framework to the three
	concrete PDEs:  the Schr\"odinger equation with convolution
	potential, the fractional Schr\"odinger equation, and the beam
	equation with metric. We verify Assumptions~1--2 for these PDEs to suit our abstract theorem.

	\section{Setting}\label{sec:2}
	We denote by the index set $\mathcal{Z}=\mathbb{Z}^\mathsf{d}\times\{-1,1\}$ and, for $J=(j,\sigma)\in\mathcal{Z}$ and $c>0,$
	$$|J|^2:=|j|^2=\sum_{l=1}^{\mathsf{d}}|j_l|^2,\langle j\rangle=\max\{|j|,c\}.$$
	
	In this paper, we are mainly concerned with nearly integrable Hamiltonian
	\begin{equation}\label{Hamiltonian}
		H=H_0+P,H_0=\sum_{j\in\mathbb{Z}^\mathsf{d}}\omega_{j}|u_j|^2,P\in\mathcal{P}_{3,\infty},
	\end{equation}
	on the following infinite-dimensional Banach space:
	$$W_s=\{u=(u_J)_{J\in\mathcal{Z}},u_J\in\mathbb{C}\mid\Vert u\Vert_s:=\sum_{j\in\mathcal{Z}}|u_j|^2\e^{2sf(\langle j\rangle)}<\infty\}. $$
	
	Here, $f$ satisfies following condition:
	\begin{assumption}
		\begin{enumerate}	
			\textbf{A.0} Weight function $f$ satisfies the followings.	
			\item $f:\mathbb{N}^+\to\mathbb{R}^+$;
			\item $f$ is a monotonically increasing function tending to $+\infty$;
			\item There exists a constant $C_f<1$ satisfying $f(\sum_{l=1}^d x_l)\leq f(x_m)+C_f\sum_{l\neq m}f(x_l),$ where $x_m=\max\{x_1,\dots,x_d\},\forall x_l\geq c$.
		\end{enumerate}
	\end{assumption}

	Two typical function classes are contained in this assumption which are infinitely differentiable but non-analytic functions: the Gevrey class and the ultra-differentiable function class.
	If $f$ is taken as $f(x)=x^\theta,0<\theta<1$, the weighting corresponds to the Gevrey class function space which we denote by $W^G_{s,\theta}$. If taken as $f(x)=(\ln x+\kappa)^q$, where the constant $\kappa$ can be adjust to satisfy the condition of $f$, it corresponds to the ultra-differentiable function space, which we denote as $W^U_{s,q}$.
	The following discussion up to the Normal Form Lemma will consistently use the abstract weighted space $W_s$ as the basis for discussion.
	
	Besides, we denote the ball in $W_s$ centered at the origin with radius $r$ by $B_s(r)$.
	For a functional $H$ defined on the space $W_s$, it determines a Hamiltonian system
	$$\dot{u}_{(j,+1)}=-\i\frac{\partial H}{\partial u_{(j,-1)}},\quad 
	\dot{u}_{(j,-1)}=\i\frac{\partial H}{\partial u_{(j,+1)}}. $$
	By denoting $\overline{J}=(j,-\sigma)$ for $J=(j,\sigma)$, we can also denote the corresponding vector field:
	$$X_H(u):=(X_J)_{J\in\mathcal{Z}},\quad (X_H)_{(j,\sigma)}=-\sigma \i\frac{\partial H}{\partial u_{\overline{(j,\sigma)}}}.$$
	%Hamiltonian system,vector feild
	For $d$-degree monomials $M=\prod_{l=1}^{d}u_{J_l},J_l=(j_l,\sigma_l)$, we denote its multi-index $\mathcal{J}=(J_1,...,J_d)$ and its momentum indicator $$\mathcal{M}_d(\mathcal{J})=\sum_{l=1}^{d}\sigma_lj_l.$$ 
	
	We will focus primarily on the monomials and polynomials whose multi-indices are in the following sets
	$$\mathcal{I}_d=\{\mathcal{J}\in\mathcal{Z}^d\mid\mathcal{M}_d(\mathcal{J})=0\},$$ which means momentum conservation.
	
	For a homogeneous polynomial $P$ of degree $d$, it can be written in the form
	\begin{equation}\label{P}
		P(u)=\sum_{J_1,\dots,J_d\in\mathcal{Z}}P_{J_1,\dots J_d}u_{J_1}\dots u_{J_d}.
	\end{equation}
	If we denote $\{J_1,\dots,J_d\}=\mathcal{J},$ we also denote $P(u)=\sum_{\mathcal{J}\in\mathcal{Z}^d}P_{\mathcal{J}}u^\mathcal{J}.$
	We are now ready to define the functional class under consideration
	\begin{definition}
		Let $d\geq1.$ We denote by $\mathcal{P}_{d}$ the space of formal polynomials $P(u)$ of the form \eqref{P} satisfying the following conditions:
		\begin{enumerate}
			\item Momentum conservation:
			$P(u)$ contains only monomials with $0$ momentum indicator, namely $$P(u)=\sum_{\mathcal{J}\in\mathcal{I}_d}P_{\mathcal{J}}u_{J_1}\dots u_{J_d};$$
			\item Reality:
			for any $\mathcal{J}\in\mathcal{Z}^d$, we have $\overline{P_{\mathcal{J}}}=P_{\overline{\mathcal{J}}}$;
			\item Boundedness:
			$$C_P:=\sup_{\mathcal{J}\in\mathcal{I}_d}|P_{\mathcal{J}}|<\infty.$$
		\end{enumerate}
	\end{definition}
	For given $r,s>0$, we can endow the space $\mathcal{P}_d$ with the norm:
	$$|P|_{r,s}:=\frac{1}{r}\sup_{u\in B_s(r)}\Vert X_{P}\Vert_s.$$
	For given integers $\infty> d_2\geq d_1\geq1$, we denote by $\mathcal{P}_{d_1,d_2}:=\bigcup_{k=d_1}^{d_2}\mathcal{P}_k$ the space of polynomials $P(u)$ that may be written as
	$$P=\sum_{k=d_1}^{d_2}P_k,P_k\in\mathcal{P}_k,$$
	endowed with the same norm
	$$|P|_{r,s}:=\frac{1}{r}\sup_{u\in B_s(r)}\Vert X_{P}\Vert_s.$$
	Similarly, we can define $\mathcal{P}_{d,\infty}=\bigcup_{k\geq d}\mathcal{P}_k$. Since $P\in\mathcal{P}_{d,\infty}$ can be written as 
	$$P=\sum_{k\geq d}P_k,P_k\in\mathcal{P}_k,$$
	the norm of $\mathcal{P}_{d,\infty}$ is the same as above. When $d_1>d_2$, we define $\mathcal{P}_{d_1,d_2}:=\emptyset.$
	
	For $P_1,P_2\in \mathcal{P}_{d_1,d_2},$ we define their Poisson brackets by
	$$\{P_1,P_2\}:=-\i\sum_{(j,\sigma)\in\mathcal{Z}}\sigma\frac{\partial P_1}{\partial u_{(j,\sigma)}}\frac{\partial P_2}{\partial u_{(j,-\sigma)}}.$$
	
	For a positive integer $N,$ we can divide the index into two cases: high mode $|J|> N$ and low mode $|J|\leq N$. Then $u\in W_s$ can also decompose by index case 
	$$u=u^{>}+u^{<}:=\sum_{|J|>N}u_{J}+\sum_{|J|\leq N}u_{J}.$$
	Then we can define a projector $\Pi^>(u):=u^>$ for $u\in W_s$. 
	In this way, we can classify polynomials based on the degree of vanishing at $0$ with respect to $u^{>}.$
	
	The constants with text as subscripts in this paper will be provided in Appendix A, constants with numeric subscripts are pure constants, while constants with variable subscripts solely depend on these variables and do not affect the main conclusion.
	
	\section{Resonance and Normal Form}\label{sec:3}
	Consider the frequencies in
	$$H_0=\sum_{j\in\mathbb{Z}^d}\omega_{j}|u_j|^2,$$
	and we demand the following assumptions:
	
	\begin{assumption}
		The frequency $(\omega_{j})_{j\in\mathbb{Z}^\mathsf{d}}$ satisfies the following properties:
		
		\textbf{A.1} There exist constants $C_0$ and $\beta>1$ such that, for sufficiently large $j$, we have
		$$\frac{1}{C_0}\leq\frac{\omega_{j}}{|j|^\beta}\leq C_0.$$

		\textbf{A.2}  Any finite component of $(\omega_{j})_{j\in\mathbb{Z}^\mathsf{d}}$ is Diophantine, namely, for $N$ large enough, $\forall J_1,\dots,J_d $ with $|J_l|\leq N,l=1,\dots,d$, and $\sum_{l=1}^{d}\omega_{j_l}\sigma_l\neq0$, we have 
		\begin{equation}\label{nr}
			\left|\sum_{l=1}^{d}\omega_{j_l}\sigma_l\right|\geq\frac{\gamma}{N^{\tau d^p}}.
		\end{equation}
		And when $\sum_{l=1}^{d}\omega_{j_l}\sigma_l=0,$ it must imply that $d$ is even and that there exists a permutation of $(1,\dots,d)$ such that
		$$\forall i=1,\dots d/2,\omega_{j_{\tau(i)}}=\omega_{j_{\tau (i+d/2)}}\  \text{and}\  \sigma_{\tau(i)}=\sigma_{\tau(i+d/2)}.$$
		This is the vital assumption for non-resonance of $(\omega_{j})_{j\in\mathbb{Z}^d}$.
		
		\textbf{A.3} We define a block division for index set: $\mathcal{Z}=\bigcup_{\alpha}\Omega_{\alpha}$.
		The division satisfies
		\begin{enumerate}
			\item There exists a $\Omega_0$ satisfying $\forall J\in\Omega_0,|J|\leq C_0$;
			\item $\forall\alpha\neq0,$ there exists a constant $C_1$ such that $\sup_{J\in\Omega_{\alpha}}|J|-\inf_{J\in\Omega_{\alpha}}|J|\leq C_1$;
			\item $\forall j_1\in\Omega_{\alpha},j_2\in\Omega_\beta,\alpha\neq\beta,|\omega_{j_1}-\omega_{j_2}|\geq C_2(|j_1|^\delta+|j_2|^\delta),$
				where $\delta>0.$	
		\end{enumerate}	

	\end{assumption}

	For a multi-index $(J_1,J_2,...,J_d)=\mathcal{J}\in\mathcal{Z}^d$, if we denote by $\mathsf{s}$ the number of components satisfying $|J_i|>N$, we can illustrate the division of the index set using the following flowchart:
	
	
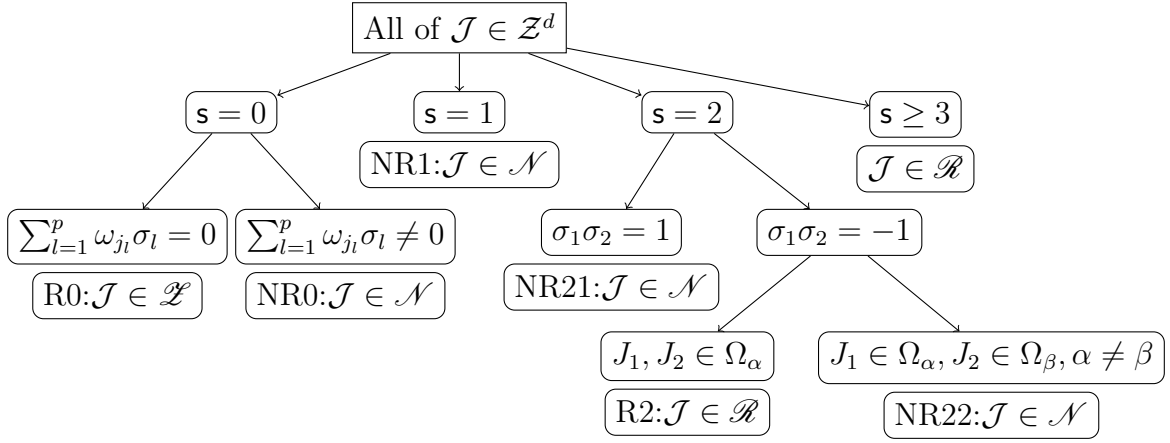
\begin{center}
		\begin{tikzpicture}[node distance=1cm,auto]
			\node (start) [draw, rectangle]{All of $\mathcal{J}\in\mathcal{Z}^d$};
			\node (case0) [draw, rectangle, rounded corners, align=center, below=0.5cm of start, xshift=-3cm] {$\mathsf{s}=0$};
			\node (case1) [draw, rectangle, rounded corners, align=center, below=0.5cm of start, xshift=0cm] {$\mathsf{s}=1$};
			\node (case2) [draw, rectangle, rounded corners, align=center, below=0.5cm of start, xshift=3cm] {$\mathsf{s}=2$};
			\node (case3) [draw, rectangle, rounded corners, align=center, below=0.5cm of start, xshift=6cm] {$\mathsf{s}\geq3$};
			
			\node (case01) [draw, rectangle, rounded corners, align=center, below=1cm of case0, xshift=-1.5cm] {$\sum_{l=1}^{p}\omega_{j_l}\sigma_l=0$};
			\node (R0) [draw, rectangle, rounded corners, below=0.1cm of case01]{R0:$\mathcal{J}\in\mathscr{Z}$};
			\node (case02) [draw, rectangle, rounded corners, align=center, below=1cm of case0, xshift=1.5cm] {$\sum_{l=1}^{p}\omega_{j_l}\sigma_l\neq0$};
			\node (NR02) [draw, rectangle, rounded corners, below=0.1cm of case02]{NR0:$\mathcal{J}\in \mathscr{N}$};
			\node (NR1) [draw, rectangle, rounded corners, below=0.1cm of case1]{NR1:$\mathcal{J}\in\mathscr{N}$};
			\node (3) [draw, rectangle, rounded corners, below=0.1cm of case3]{$\mathcal{J}\in\mathscr{R}$};
			
			\node (case21) [draw, rectangle, rounded corners, align=center, below=1cm of case2, xshift=-1cm] {$\sigma_1\sigma_2=1$};
			\node (case22) [draw, rectangle, rounded corners, align=center, below=1cm of case2, xshift=2cm] {$\sigma_1\sigma_2=-1$};
			
			\node (case221) [draw, rectangle, rounded corners, align=center, below=1cm of case22, xshift=-2cm] {$J_1,J_2\in\Omega_\alpha$};
			\node (case222) [draw, rectangle, rounded corners, align=center, below=1cm of case22, xshift=2cm] {$J_1\in\Omega_\alpha,J_2\in\Omega_\beta,\alpha\neq\beta$};		
			\node (NR21) [draw, rectangle, rounded corners, below=0.1cm of case21]{NR21:$\mathcal{J}\in\mathscr{N}$};
			\node (R221) [draw, rectangle, rounded corners, below=0.1cm of case221]{R2:$\mathcal{J}\in\mathscr{R}$};
			\node (NR222) [draw, rectangle, rounded corners, below=0.1cm of case222]{NR22:$\mathcal{J}\in\mathscr{N}$};	 		
			
			\draw[->] (start) -- (case0);
			\draw[->] (start) -- (case1);
			\draw[->] (start) -- (case2);	
			\draw[->] (start) -- (case3);
			\draw[->] (case0) -- (case01);
			\draw[->] (case0) -- (case02);
			\draw[->] (case2) -- (case21);
			\draw[->] (case2) -- (case22);
			\draw[->] (case22) -- (case221);
			\draw[->] (case22) -- (case222);	 		
			
		\end{tikzpicture}			 
	\end{center}
	All multi-indices are divided into: resonant set $\mathscr{Z}$ (containing cases R0 and R2), non-resonant set $\mathscr{N}$ (containing cases NR0, NR1, NR21, and NR22), and high mode set $\mathscr{R}$.
	
	\begin{definition}[$N$-cutting normal form]
		For given integers $N\gg1,d\geq3$, we say that a polynomial $Z\in\mathcal{P}_{3,d}$ of the form:
		$$Z=\sum_{l=3}^{d}\sum_{\mathcal{J}\in\mathcal{I}_k}Z_{\mathcal{J}}u_{J_1}...u_{J_k}$$
		is a $N$-cutting normal form, if $\mathcal{J}=\{J_1,\dots,J_k\}$ in one of the following cases:
		\begin{enumerate}
			\item For every $1\leq l\leq k$, $|J_l|\leq N $ and $\sum_{l=1}^{k}\sigma_l\omega_{j_l}=0 $;
			\item There exactly exist $J_1=(j_1,\sigma_1),J_2=(j_2,\sigma_2),|J_1|>N,|J_2|>N,\sigma_1\sigma_2=-1$, and $J_1,J_2$ in the same $\Omega_\alpha$. 
		\end{enumerate}
		Namely $\mathcal{J}\in\mathscr{Z}.$
	\end{definition}

	\section{Iteration lemma}\label{sec:4}
	
	We first give an estimate of the solution of the homological equation
	$$\{H_0,G\}+P=Z.$$ 
	\begin{lemma}[Homological equation]\label{homo}
		For $P\in\mathcal{P}_{k,p}$ of at most  degree $2$ with respect to $u^>$ and $H_0$'s frequency satisfy \textbf{Assumption 2} , there exists a $G\in\mathcal{P}_{k,p}$ solving the following equation
		$$P=\sum_{\mathcal{J}\in\mathscr{Z}} P_\mathcal{J}u^\mathcal{J}+\sum_{\mathcal{J}\in\mathscr{N}} P_\mathcal{J}u^\mathcal{J}=Z-\{H_0,G\},$$
		with the estimates:
		\begin{enumerate}
			\item $|Z|_{r,s}\leq|P|_{r,s}$;
			\item $|G|_{r,s}\leq \frac{1}{\gamma}(C_{\mathtt{deno}}dN)^{C_{\mathtt{exp}}d^p}|P|_{r,s}.$
		\end{enumerate}
	\end{lemma}
	
	\begin{proof}
		Since $Z$ is composed of only a part of $P$, its estimate is straightforward.
		
		Assume $G=\sum_{\mathcal{J}\in\mathcal{N}}G_{\mathcal{J}}u^{\mathcal{J}},$ then $$-\{H_0,G\}=\sum_{\mathcal{J}\in\mathcal{N}}\sum_{l=1}^{d}\omega_{j_l}\sigma_lG_{\mathcal{J}}u^{\mathcal{J}}=\sum_{\mathcal{J}\in\mathcal{N}}P_{\mathcal{J}}u^{\mathcal{J}}.$$
		Now we mainly need to estimate the small denominator $\sum_{l=1}^{d}\omega_{j_l}\sigma_l$.
		
		According to the flow chart, case NR0 consists of all indices satisfying $|J_l|\leq N$ and $\sum_{l}\sigma_l\omega_{j_l}\neq0,$ then it directly follows from (\ref{nr}) in A.2  $\left|\sum_{l=1}^{d}\omega_{j_l}\sigma_l\right|\geq\frac{\gamma}{N^{\tau d^p}}.$ 
		
		For case NR1, we have the following inequalities from A.1 $$\left|\sum_{l\neq1}\sigma_l\omega_{j_l}\right|\leq (d-1)C_0N^\beta,\quad |\omega_{j}|\geq \frac{1}{C_0}|j_1|^\beta.$$
		So, when $|j_1|\geq d^{\frac{1}{\beta}}C_0^{\frac{1}{\beta}}N:=N_1$, 
		$$\left|\sum_{l=1}^{d}\omega_{j_l}\sigma_l\right|\geq C_0N^\beta>1.$$
		When $|j_l|\leq|j_1|\leq N_1,l\neq1,$ we can get $$\left|\sum_{l=1}^{d}\omega_{j_l}\sigma_l\right|\geq\frac{\gamma}{N_1^{\tau d^p}}=\frac{\gamma}{(d^{\frac{1}{\beta}}C_0^{\frac{2}{\beta}}N)^{\tau d^p}}$$
		from A.2.
		
		For case NR21, we also have the following inequalities from A.1
		$$\left|\sum_{l\neq1,2}\sigma_l\omega_{j_l}\right|\leq (d-2)C_0N^\beta,\quad |\omega_{j_1}|+|\omega_{j_2}|\geq \frac{1}{C_0}|\max\{|j_1|,|j_2|\}|^\beta.$$ 
		So, when $\max\{|j_1|,|j_2|\}\geq d^{\frac{1}{\beta}}C_0^{\frac{2}{\beta}}N=N_1$, 
		$$\left|\sum_{l=1}^{d}\omega_{j_l}\sigma_l\right|\geq2C_0N^\beta>1.$$
		When $|j_l|\leq\max\{|j_1|,|j_2|\}\leq N_1,l\neq1,2$, we also can get
		$$\left|\sum_{l=1}^{d}\omega_{j_l}\sigma_l\right|\geq\frac{\gamma}{N_1^{\tau d^p}}=\frac{C_{p,\beta}}{(d^{\frac{1}{\beta}}C_0^{\frac{2}{\beta}}N)^{\tau d^p}}$$
		from A.2.
		
		For case NR22, we use A.3 to get
		$$\left|\sum_{l\neq1,2}\sigma_l\omega_{j_l}\right|\leq (d-2)C_0N^\beta,\quad |\omega_{j_1}-\omega_{j_2}|\geq C_2(|j_1|^\delta+|j_2|^\delta) .$$
		Thus, when $|j_1|^\delta+|j_2|^\delta\geq \frac{C_0}{C_2}dN^\beta:=N_2^\delta,$
		$$\left|\sum_{l=1}^{p}\omega_{j_l}\sigma_l\right|\geq 2C_0N^{\beta}>1.$$  
		When $|j_1|^\delta+|j_2|^\delta\leq \frac{C_0}{C_2}dN^\beta=N_2^{\delta},|j_l|\leq N,l\neq1,2$, we also have
		$$\left|\sum_{l=1}^{d}\omega_{j_l}\sigma_l\right|\geq\frac{\gamma}{N_2^{\tau d^p}}=\frac{\gamma}{(\frac{C_0}{C_2}dN^\beta)^{\frac{\tau}{\delta} d^p}}$$
		from A.2.
		
	From the above analysis, we conclude that the small denominators for all non-resonant indices satisfy the estimate:
	
		$$   \left|\sum_{l=1}^{d}\omega_{j_l}\sigma_l\right|\geq\frac{\gamma}{(C_{\mathtt{deno}}dN)^{C_{\mathtt{exp}}d^p}}.$$
		Therefore, $$|G|_{r,s}\leq \frac{1}{\gamma}(C_{\mathtt{deno}}dN)^{C_{\mathtt{exp}}d^p}|P|_{r,s}.$$ 	 
	\end{proof}
	Now we can begin the iterate process.

	\begin{lemma}[Iteration lemma]\label{iter}
		For $H$ as defined in (\ref{Hamiltonian}), let $d\geq k\geq3,d>D_{\mathtt{fin}},s>s_0>\frac{\mathsf{d}}{2}$. Then for parameters satisfying the conditions $$rd^2C_{\mathtt{thre}}(C_{\mathtt{deno}}dN)^{C_{\mathtt{exp}}d^p}<1,\quad r_{k}=2r-\frac{(k-3)r}{d-3},$$ there exists a sequence of transformations $\mathcal{T}^{(k)}:B^s(r_k)\to B^s(r_3)$ satisfying the following properties:
		\begin{enumerate}
			\item $H^{(k)}:=H\circ\mathcal{T}^{(k)}=H_0+Z_{k}+P_k+R_{k,d}+R_{k,>}$;
			\item $P_k\in\mathcal{P}_{k,d},|P_k|_{r_k,s}\leq d^{2k-7}(C_{\mathtt{estP}}r)^{k-2}(C_{\mathtt{deno}}dN)^{C_{\mathtt{exp}}(k-3)d^p}  $;
			\item 
			$|R_{k,d}|_{r_k,s}\leq r^{d-2}(C_{\mathtt{rema}}dN)^{C_{\mathtt{exp}}d^{p+1}}$;
			\item 
			$|R_{d,>}|_{r,s}\leq \frac{C_{R}r}{\e^{(s-s_0)f(N)}}$.
		\end{enumerate}  
	\end{lemma}
	\begin{proof}
		For the initial $H=H_0+P,$ and given integer $N,$ we can decompose $P=P_3+R_{3,>}+R_{3,d}$, where $R_{3,d}\in\mathcal{P}_{d+1,\infty}$, $R_{3,>}$ is at least three degree for $u^>$. Then we have $$H^{(3)}:=H_0+P_3+R_{3,>}+R_{3,d}.$$
		Now we consider the following homological equations for $d\geq k\geq3$:
		$$\{H_0,G_{k+1}\}+P_k=Z_{k}^*,$$  
		and define $Z_3:=Z_3^*.$ 
		Then we can discuss the time-$1$ map of Hamiltonian flow generated by $G_{k+1}$ action on every term of $H^{(k)}$:
		\begin{align*}
			H^{(k+1)}:=H^{(k)}\circ\Phi_{G_{k+1}} =(H_0+Z_{k}+P_{k}+R_{k,d}+R_{k,>})\circ\Phi_{G_{k+1}}.
		\end{align*}
		
		Consider the transformation acting on term $H_0$. Let $n$ be an integer such that $n(k-2)+k>d$, then
		\begin{align*}
			H_0\circ\Phi_{G_{k+1}} &=H_0+\{H_0,G_{k+1}\}+\sum_{l=2}^{n}\frac{ad^l_{G_{k+1}}}{l!}H_0+R_{H_0,G_{k+1}}\\
			&=H_0+\{H_0,G_{k+1}\}+P_{H_0,k+1}+R_{H_0,G_{k+1}}.
		\end{align*}
		By the choice of $n$, we can obtain $P_{H_0,k+1}\in\mathcal{P}_{k+1,d},R_{H_0,G_{k+1}}\in\mathcal{P}_{d+1,\infty}$.
		%Because  from Lemma \ref{homo}.
		
		Consider the transformation acting on term $Z_k$:
		\begin{align*}
			Z_k\circ\Phi_{G_{k+1}}&=Z_k+\sum_{l=1}^{n}\frac{ad^l_{G_{k+1}}}{l!}Z_k+R_{Z_k,G_{k+1}}\\
			&=Z_k+P_{Z_k,k+1}+R_{Z_k,G_{k+1}},
		\end{align*}
		where $n$ is taken the same way as above, so there is $P_{Z_k,k+1}\in\mathcal{P}_{k+1,d},R_{Z_k,G_{k+1}}\in\mathcal{P}_{d+1,\infty}$.
		
		Consider the transformation acting on term $P_k$:
		\begin{align*}
			P_k\circ\Phi_{G_{k+1}}&=P_k+\sum_{l=1}^{n}\frac{ad^l_{G_{k+1}}}{l!}P_k+R_{P_k,G_{k+1}}\\
			&=P_k+P_{P_k,k+1}+R_{P_k,G_{k+1}}.
		\end{align*}
		There is also $P_{P_k,k+1}\in\mathcal{P}_{k+1,d},R_{P_k,G_{k+1}}\in\mathcal{P}_{d+1,\infty}$.
		
		Consider the transformation acting on term $R_{k,d}$. Note that $R_{k,d}\circ\Phi_{G_{k+1}}$ has at least $d+1$ order zero at $u=0$. Hence
		$$R_{k,d}\circ\Phi_{G_{k+1}}\in\mathcal{P}_{d+1,\infty}.$$
		
		Consider the transformation acting on term $R_{k,>}$. Now
		$R_{k,>}\circ\Phi_{G_{k+1}}$ has at least a third order zero at $u^>=0$.
		
		Summing up the above, we can rearrange as follows for the $H^{(k+1)}$:
		\begin{align*}
			H^{(k+1)}:&=H_0+(\{H_0,G_{k+1}\}+P_k)+Z_k+\\
			&+P_{H_0,k+1}+P_{Z_k,k+1}+P_{P_k,k+1}\\
			&+R_{H_0,G_{k+1}}+R_{Z_k,G_{k+1}}+R_{P_k,G_{k+1}}+R_{k,\bar{p}}\circ\Phi_{G_{k+1}}\\
			&+R_{k,>}\circ\Phi_{G_{k+1}}.
		\end{align*} 
		Define
		\begin{align*}
			R_{d,k+1}:&=R_{H_0,G_{k+1}}+R_{Z_k,G_{k+1}}+R_{P_k,G_{k+1}}+R_{k,\bar{p}}\circ\Phi_{G_{k+1}},\\
			P_{k+1}^*:&=P_{H_0,k+1}+P_{Z_k,k+1}+P_{P_k,k+1}.
		\end{align*}
		Then we make decompose $P_{k+1}^*=P_{k+1}+ R_{k+1,>}^*$, where $R_{k+1,>}^*$ includes all terms having at least 3 degree zero at $u^>=0$, $P_{k+1}\in\mathcal{P}_{k+1,\bar{p}}$ and the zero of this term with respect to $u^>$ is at most $2$ degree. Therefore, we have
		\begin{align*}
			H^{(k+1)}&=H_0+( Z_{k}^*)+Z_k+P_{k+1}^*+R_{k+1,d}+R_{k,>}\circ\Phi_{G_{k+1}}\\
			&=H_0+(Z_k+ Z_{k}^*)+P_{k+1}+R_{k+1,d}+(R_{k,>}\circ\Phi_{G_{k+1}}+ R_{k+1,>}^*)\\
			&:=H_0+Z_{k+1}+P_{k+1}+R_{k+1,d}+R_{k+1,>}. 
		\end{align*}
		To estimate the terms in $H^{(k+1)},$ we first use Lemma \ref{homo} to get
		\begin{align*}
			|G_{k+1}|_{r_{k},s}&\leq\frac{|P_k|_{r_k,s}}{\gamma}(C_{\mathtt{deno}}dN)^{C_{\mathtt{exp}}d^p},\\
			|Z_{k+1}-Z_k|_{r_k,s}&\leq| Z_{k}^*|_{r_k,s}\leq|P_k|_{r_k,s}.
		\end{align*}
		
		We use induction to prove the estimates of $P_k$ and $G_{k+1}$ during the iteration process.
		
		By the choice of $r$, we have the first inductive step
		\begin{align*}
			|G_4|_{r_3,s}&\leq\frac{|P_3|_{r_3,s}}{\gamma}(C_{\mathtt{deno}}dN)^{C_{\mathtt{exp}}d^p}\leq\frac{2C_Pr}{\gamma}(C_{\mathtt{deno}}dN)^{C_{\mathtt{exp}}d^p}\\
			&\leq E:=\frac{1}{16\e d}<\frac{r_k-r_{k+1}}{8\e r_k}.
		\end{align*} 
		Then we can use Lemma \ref{Lie} to prove the estimates for $P_{k+1},G_{k+2}$ based on the estimates for $P_{k},G_{k+1}$ inductively. Note that
		\begin{align*}
			|P_{k+1}|_{r_{k+1},s}&\leq|P_{H_0,k+1}|_{r_{k+1},s}+|P_{Z_k,k+1}|_{r_{k+1},s}+|P_{P_k,k+1}|_{r_{k+1},s}\\
			&\leq|\sum_{l=2}^{n}\frac{ad^l_{G_{k+1}}}{l!}H_0|_{r_{k+1},s}+|\sum_{l=1}^{n}\frac{ad^l_{G_{k+1}}}{l!}Z_k|_{r_{k+1},s}+|\sum_{l=1}^{n}\frac{ad^l_{G_{k+1}}}{l!}P_k|_{r_{k+1},s}\\
			&\leq|\sum_{l=1}^{n-1}\frac{ad_{G_{k+1}}^l}{(l+1)!}\{G_{k+1},H_0\}|_{r_{k+1},s}+|\sum_{l=1}^{n}\frac{ad^l_{G_{k+1}}}{l!}(Z_k+P_k)|_{r_{k+1},s}\\
			&\leq\sum_{l=1}^{n-1}\frac{1}{(l+1)!}(\frac{|G_{k+1}|_{r_k,s}}{2E})^l|P_k|_{r_k,s}\\
			&+\sum_{l=1}^{n}\frac{1}{l!}(\frac{|G_{k+1}|_{r_k,s}}{2E})^l(|P_k|_{r_k,s}+\sum_{m=3}^{k-1}|Z_{m+1}-Z_m|_{r_k,s}+|Z_3|_{r_k,s})\\
			&\leq\sum_{l=1}^{n}\frac{2}{l!}(\frac{|G_{k+1}|_{r_k,s}}{2E})^l(\sum_{m=3}^{k}|P_m|_{r_m,s})\\
			&\leq \frac{\e|G_{k+1}|_{r_k,s}}{E}\sum_{m=3}^{k}|P_m|_{r_m,s}.
		\end{align*}
		When $k=3$, we have 
		$$|P_4|_{r_4,s}\leq \frac{\e}{E\gamma } C_P^2r_3^2(C_{\mathtt{deno}}dN)^{C_{\mathtt{exp}}d^p}\leq\frac{64\e^2d}{\gamma }C_P^2r^2(C_{\mathtt{deno}}dN)^{C_{\mathtt{exp}}d^p}.$$
		Therefore, when $k\geq4$, we will use induction to prove that there exists a constant $C_{\mathtt{estP}}$ such that $|P_{k}|_{r_k,s}\leq d^{2k-7}(C_{\mathtt{estP}}r)^{k-2}(C_{\mathtt{deno}}dN)^{C_{\mathtt{exp}}(k-3)d^p} $ holds for $k\geq 4$. And 
		\begin{align*}
			|P_{k+1}|_{r_{k+1},s}&\leq\frac{e}{\gamma E}d^{2k-7}(C_{\mathtt{estP}}r)^{k-2}(C_{\mathtt{deno}}dN)^{C_{\mathtt{exp}}(k-2)d^p}(\sum_{l=3}^{k}d^{2l-7}(C_{\mathtt{estP}}r)^{l-2}(C_{\mathtt{deno}}dN)^{C_{\mathtt{exp}}(l-3)d^p})\\
			&\leq \frac{16\e^2}{\gamma}d^{2k-6}(C_{\mathtt{estP}}r)^{k-2}(C_{\mathtt{deno}}dN)^{C_{\mathtt{exp}}(k-2)d^p}\frac{C_{\mathtt{estP}}r}{1-d^2C_{\mathtt{estP}}r(C_{\mathtt{deno}}dN)^{C_{\mathtt{estP}}d^p}}\\
			&\leq \frac{32\e^2}{\gamma}d^{2k-6}(C_{\mathtt{estP}}r)^{k-1}(C_{\mathtt{deno}}dN)^{C_{\mathtt{exp}}(k-2)d^p}\\
			&\leq d^{2k-5}(C_{\mathtt{estP}}r)^{k-1}(C_{\mathtt{deno}}dN)^{C_{\mathtt{exp}}(k-2)d^p}.
		\end{align*}
		Here we use the setting of $r,d$. Then we have  
		\begin{align*}
			|G_{k+2}|_{r_{k+1},s}&\leq\frac{1}{\gamma} d^{2k-7}(C_{\mathtt{estP}}r)^{k-2}(C_{\mathtt{deno}}dN)^{C_{\mathtt{exp}}(k-2)d^p}\\
			&\leq \frac{1}{d^3\gamma}(C_{\mathtt{estP}}rd^2(C_{\mathtt{deno}}dN)^{C_{\mathtt{exp}}d^p})^{k-2}\leq E
		\end{align*}
		by Lemma \ref{homo} and the setting of $r$.  
		Thus we have completed the inductive proof of the estimates for $P_k,G_{k+1}.$

		It follows from the definition of norm that $\sup_{u\in B^s(r_{k+1})}\Vert X_{G_{k+1}}\Vert_{s}\leq r_{k+1}|G_{k+1}|_{r_{k+1},s},$ which leads to the near-identity property of $\Phi_{G_{k+1}}$:
		$$\sup_{u\in B^s(r_{k+1})}\Vert (\Phi_{G_{k+1}}-Id)\circ(u)\Vert_s\leq\int_{0}^{1}\sup_{u\in B^s(r_{k+1})} \Vert X_{G_{k+1}}(u(T))\Vert_sdT\leq r_{k+1}E\leq r_{k}-r_{k+1} .$$
		Namely the transformation maps $B^s(r_{k+1})$ into $B^s(r_k)$.
		
		Besides, from the integral-type remainder  $$R_{X,G_{k+1}}=\frac{1}{n!}\int_{0}^{1}(1-T)^n(ad_{G_{k+1}}^{n+1}X)\circ\Phi_{G_{k+1}}^TdT,X=\{G_{k+1},H_0\},Z_k,P_k,$$
		we get 
		\begin{align*}
			|R_{X,G_k+1}|_{r_k,s}&\leq\frac{1}{n!}|X|_{r_k,s}(\frac{|G_{k+1}|}{2E})^n\\
			&\leq d^{2k-7}(C_{\mathtt{estP}}r)^{k-2}(C_{\mathtt{deno}}dN)^{C_{\mathtt{exp}}(k-3)d^p}(\frac{8\e d}{\gamma} d^{2k-7}(C_{\mathtt{estP}}r)^{k-2}(C_{\mathtt{deno}}dN)^{C_{\mathtt{exp}}(k-2)d^p})^n  \\
			&\leq  d^{2k-7+n(2k-4)}(C_{\mathtt{estP}}r)^{(k-2)(n+1)}(C_{\mathtt{deno}}dN)^{C_{\mathtt{exp}}(n(k-2)+k-3)d^p}(\frac{8\e}{d^2\gamma})^n           \\
			&\leq d^{2d-7} (C_{\mathtt{estP}}r)^{d-2}(C_{\mathtt{deno}}dN)^{C_{\mathtt{exp}}(d-3)d^p}
			%&\leq r^{d-2}(\max\{C_{\mathtt{estP}},C_{\mathtt{deno}}\}dN)^{C_{\mathtt{exp}}d^{p+1}}
		\end{align*}
		by Lemma \ref{Lie}. Then by the iterative process involving $R_{k+1,d}$,
		\begin{align*}
			|R_{k+1,d}|_{r_{k+1},s}&\leq |R_{H_0,G_k+1}|_{r_k,s}+|R_{Z_k,G_k+1}|_{r_k,s}+|R_{P_k,G_k+1}|_{r_k,s}+|R_{k,d}\circ\Phi_{G_{k+1}}|_{r_k,s}\\
			&\leq 3d^{2d-7} (C_{\mathtt{estP}}r)^{d-2}(C_{\mathtt{deno}}dN)^{C_{\mathtt{exp}}(d-3)d^p}+(1+E)|R_{k,d}|_{r_k,s},\\
			\frac{|R_{k+1,d}|_{r_{k+1},s}}{(1+E)^{k+1}}&\leq\frac{3}{(1+E)^{k+1}}d^{2d-7} (C_{\mathtt{estP}}r)^{d-2}(C_{\mathtt{deno}}dN)^{C_{\mathtt{exp}}(d-3)d^p}+\frac{|R_{k,d}|_{r_k,s}}{(1+E)^k},\\
			|R_{k,d}|_{r_k,s}&\leq3\frac{(1+E)^{k-3}-1}{E}d^{2d-7} (C_{\mathtt{estP}}r)^{d-2}(C_{\mathtt{deno}}dN)^{C_{\mathtt{exp}}(d-3)d^p}+(1+E)^{k-3}C_Pr^{d-2}\\
			&\leq48\e d(\e^{\frac{1}{16\e}}-1)d^{2d-7} (C_{\mathtt{estP}}r)^{d-2}(C_{\mathtt{deno}}dN)^{C_{\mathtt{exp}}(d-3)d^p}+eC_Pr^{d-2}\\
			&\leq  r^{d-2}(C_{\mathtt{rema}}dN)^{C_{\mathtt{exp}}d^{p+1}}.
		\end{align*} 
		We thus derive the estimate of $R_{k,d}$.
		
		Since $R_{k,>}\in\mathcal{P}_{3,\infty}$, we can use Lemma \ref{cut} and the choice of $r$ to derive
		$|R_{d,>}|_{r_k,s}\leq\frac{rC_{R}}{\e^{(s-s_0)f(N)}}.$
		
		Finally, the transformation $\mathcal{T}^{(k)}=\Phi_{G_{4}}\circ...\circ\Phi_{G_{k+1}}:B_s(r_k)\to B_s(r_{3})$ is the desired transformation.
		
	\end{proof}

	\section{Normal Form Lemma}\label{sec:5}
	In this section, we balance the order of the two remaining terms in the Iteration Lemma to obtain the Normal Form Lemma to be used.
	\begin{theorem}[Normal Form Lemma]\label{Normal}
		For $H$ as defined in (\ref{Hamiltonian}), let $d\gg3,P\in\mathcal{P}_{3,\infty},$ then there exist $N_{d}$ and a canonical transformation $\mathcal{T}_{d}$ such that for $s>s_0,$ 
		the following holds for any sufficiently small $r$:
		\begin{align*}
			&\mathcal{T}_{d}:W_s(r)\to W_s(2r),\\
			&\mathcal{T}_{d}^{-1}:W_s(2r)\to W_s(r),\\
			&H^{(d)}:=H\circ\mathcal{T}_d=H_0+Z_d+R_d,
		\end{align*}	 
		where
		\begin{enumerate}
			\item $Z_d\in\mathcal{P}_{3,d}$ is in the $N-$cutting normal form;
			\item $|R_d|_{r,s}\leq \e^{-C_{\mathtt{fin}}f(N(r))}$.
			
		\end{enumerate}
		The relationship between $N(r)$ and $r$ will be implicitly provided in the proof. 
		
		Besides, there is a split $Z_d=Z_0+Z_{>}$, such that the index in $Z_0$ is in the case R0 and the index in $Z_>$ is in the case NR2, 
		and $$\sup_{u\in B_s(r)}\Vert(Id-\Pi^>)(X_{Z_>})\Vert_{s}\leq \e^{-C_{\mathtt{fin}}f(N(r))}. $$
		
	\end{theorem}
	
	\begin{proof}
		First, we set $k=d$ in Iteration Lemma \ref{iter}, the Hamiltonian comes to $H'=H_0+Z_{d}+R_{d,d}+R_{d,>}$ with $$|R_{d,d}|_{r_k,s}\leq r^{d-2}(C_{\mathtt{rema}}dN)^{C_{\mathtt{exp}}d^{p+1}},|R_{d,>}|\leq C_R\frac{r}{\e^{(s-s_0)f(N)}}.$$ 
		Next, we will adjust the parameters in the estimate of $R_{d,d}$ and $R_{d,>}$ to make them have equally order small, combining them into a single remainder term. When $d>D_{\mathtt{fin}},s>S_{\mathtt{fin}},rC_R<1,d'=\frac{d}{2}$, the above estimate simplifies
		\begin{align*}
			|R_{d,d}|_{r_k,s}&\leq r^{d-2}(C_{\mathtt{rema}}dN)^{C_{\mathtt{exp}}d^{p+1}}\\
			&\leq r^{\frac{d}{2}}(C_{\mathtt{rema}}dN)^{C_{\mathtt{exp}}d^{p+1}}\\
			&\leq r^{d'}(2C_{\mathtt{rema}}d'N)^{2^{p+1}C_{\mathtt{exp}}d'^{p+1}}\\
			&\leq r^{d'}(d'N)^{2^{p+2}C_{\mathtt{exp}}d'^{p+1}}\\
			&\leq r^{2^{p+2}C_{\mathtt{exp}}d'}(d'N)^{2^{p+2}C_{\mathtt{exp}}d'^{p+1}}.
		\end{align*}
		At the same time,
		$$|R_{d,>}|\leq \e^{-2^{p+2}C_{\mathtt{exp}}f(N)}.$$ 
		For the sake of simplicity, we continue to denote $r',d'$ by $r,d$. Now we impose the condition 
		$$(r(dN)^{d^p})^d=\e^{-f(N)},$$
		namely
		\begin{equation}\label{order}
			d^p\ln(dN)+\ln r=\frac{-f(N)}{d}.
		\end{equation}
		Making the order be same, we let 
		$$d^p\ln(dN)=\frac{f(N)}{d}.$$
		Then specifying $f(x)$, we can derive the dependency of $N$ on $d$ and substitute it back to \eqref{order} to determine its dependency on $r$. Therefore the condition of Lemma \ref{iter} reduces to requiring $r$ sufficiently small.

		Finally because of cutting lemma and definition of $Z_>$, $Z_>$ has the same order with $R_{d}$.
		
	\end{proof}

	We now present the order of the remainder in Normal Form Lemma with respect to $r$ under two representative cases of $f(x)$ as described in the following propositions.

	\begin{proposition}
		When $f(x)=x^\theta,$ with $\theta<1,\e^{f(N)}$ has higher order than $\exp(|\ln r|^{1+a})$ for $1-ap>0$. Specifically, when $p=1$, $\e^{f(N)}$ is asymptotic to $\e^{C\frac{|\ln r|^2}{\ln|\ln r|}}.$ \\
		
	\end{proposition}
	\begin{proof}
		When $f(x)=x^\theta,$ we proceed with the following calculations
		\begin{align*}
			d^{p}(\ln dN)&=\frac{N^\theta}{d}=\frac{|\ln r|}{2},  \\
			d^{p+\theta+1}(\ln dN)&=d^\theta N^\theta,N':=N^\theta,d':=d^\theta,\\
			\frac{1}{\theta}d'^{\frac{\theta+p+1}{\theta}}\ln(d'N')&=N'd':=e^D,\\
			\frac{1}{\theta}d'^{\frac{\theta+p+1}{\theta}}D&=e^{D},\\
			-\theta d'^{-\frac{\theta+p+1}{\theta}}&=-De^{-D},\\
			\theta\ln Nd=\ln N'd'=D&=-W_{-1}(\frac{-\theta}{d^{p+\theta+1}}),\\
			N&=d^{-1}e^{-\frac{1}{\theta}W_{-1}(-\theta d^{-(\theta+p+1)})},\\
			|\ln r|&=2d^{-1-\theta}e^{-W_{-1}(-\theta d^{-(\theta+p+1)})},\\
			\ln|\ln r|&=-(1+\theta)\ln d-W_{-1}(-\theta d^{-(\theta+p+1)}).
		\end{align*}	
		We first use $|\ln r|^{1+a}$ to  probe the order of $f(N)$, and in the calculation we omit the multiplicative constant:
		\begin{align*}
			\frac{N^\theta}{|\ln r|^{1+a}}&=d^{a\theta+a+1}e^{aW_{-1}(-\theta d^{-(\theta+p+1)})}\\
			&=d^{a\theta+a+1}(\frac{\theta d^{-(\theta+p+1)}}{-W_{-1}(-\theta d^{-(\theta+p+1)})})^{a}\\
			&=d^{1-ap}(-W_{-1}(-\theta d^{-(\theta+p+1)}))^{-a}\\
			\text{(same order to)}	&= \frac{d^{1-ap}}{((\theta+p+1)\ln d)^a}.
		\end{align*}
		The last line is from Lemma \ref{Lam}. So we have $e^{f(N)}\gg e^{|\ln r|^{1+a}},$ when $1-ap>0.$ When $p=1,a=1,$ we calculate
		\begin{align*}
			\frac{N^{\theta}\ln|\ln r|}{|\ln r|^2}&=\frac{-(1+\theta)\ln d-W_{-1}(-\theta d^{-(\theta+p+1)})}{\ln d}\\
			&=-(1+\theta)-\frac{W_{-1}(-\theta d^{-(\theta+p+1)})}{\ln d}.
		\end{align*}
		By Lemma \ref{Lam}, we can draw that $N^\theta$ has the same order to $\frac{|\ln r|^2}{\ln|\ln r|}.$
		
	\end{proof}
	
	\begin{proposition}
		When $f(x)=(\ln (x+\kappa))^q,$ with $q>1,e^{f(N)}$ grows faster than $e^{|\ln r|^{1+a}}$ for $a\leq\frac{q-1}{qp+1}$.
	\end{proposition}
	\begin{proof}
		When $f(x)=(\ln (x+\kappa))^q,$ we omit the constant $\kappa$ here and proceed with following calculations  
		\begin{align*}
			d^p(\ln dN)&=\frac{(\ln N)^q}{d}=\frac{|\ln r|}{2},\\
			d^{p+1} \ln N\leq d^{p+1}(\ln dN)&=(\ln N)^q\leq d^{p+1}\ln d\ln N,\\
			d^{\frac{p+1}{q-1}}&\leq\ln N\leq d^{\frac{p+1}{q-1}}(\ln d)^{\frac{1}{q-1}}.
		\end{align*}
		We then calculate 
		\begin{align*}
			\frac{(\ln N)^q}{|\ln r|^{1+a}}&=\frac{(\ln N)^q}{|\ln r|}\frac{1}{|\ln r|^a}=\frac{d^{1+a}}{(\ln N)^{qa}},    \\
			d^{1+a-qa\frac{p+1}{q-1}}(\ln d)^{-\frac{1}{q-1}}	&\leq\frac{d^{1+a}}{(\ln N)^{qa}}\leq d^{1+a-qa\frac{p+1}{q-1}}.
		\end{align*}
		Thus, for $1+a-qa\frac{p+1}{q-1}>0,e^{f(N)}$ has higher order than $\exp(|\ln r|^{1+a})$. Specifically, when $p=1,$ we get $a<\frac{q-1}{q+1}.$ 
	\end{proof}

	\section{Stability Time}\label{sec:6}
	To use the Normal Form Lemma, we set $w=\mathcal{T}^{(d)}(u),w_0=\mathcal{T}^{(d)}(u_0),$ and we consider the Cauchy problem
	\begin{equation}\label{Cau}
		\dot{w}=X_{H^{(d)}}(w),w(0)=w_0.
	\end{equation}
	Let $z(t)$ be the solution of \eqref{Cau} and define
	$$T_r:=\sup\{|t|\in\mathbb{R}^+\mid\Vert w\Vert_s\leq 2r\}$$
	as the escape time of the solution from the ball of radius $R$. Next, we split the normal form as $Z_d=Z_0+Z_>$, as stated in Theorem \ref{Normal}. We obtain the following system of equations:
	\begin{align}
		\dot{w}^<&=\Lambda w^<+X_{Z_0}(w^<)+\Pi^<X_{Z_>}(w^<,w^>)+\Pi^<X_{R_d}(w^<,w^>),\label{low}  \\
		\dot{w}^>&=\Lambda w^>+\Pi^>X_{Z_>}(w^<,w^>)+\Pi^>X_{R_d}(w^<,w^>).\label{high}
	\end{align}
	We first give a standard priori estimate on the low frequency part $w^<$ of the solution of \eqref{Cau} based on \eqref{low}. 
	\begin{proposition}
		For $s>s_0$, and any real $w_0$ with $\Vert w_0\Vert_{s}<r $ in \eqref{Cau}, we have
		$$\Vert w^<(t)\Vert_s\leq\Vert w^>(0)\Vert_{s}+e^{-C_{\mathtt{fin}}f(N(r))}|t|, \forall|t|\leq T_r.$$
	\end{proposition}
	\begin{proof}
	 Since $\{|w_{(j,+)}|^2,Z_0\}=0,$ we have
	 \begin{align*}
	 	\frac{\d}{\d t
	 	}\Vert w^<\Vert_s^2&=\{\Vert w^<\Vert_s^2,Z_2\}+\{\Vert w^<\Vert_s^2,R_d\}\\
	 	&=\sum_{J\in\mathcal{Z}}\frac{\partial}{\partial u_{J}}(\Vert w^<\Vert_s^2)\cdot(X_{Z_>}(w^<,w^>)+X_{R_d}(w^<,w^>))\\
	 	&\leq|Z_>|_{2r,s}+|R_d|_{2r,s}=\e^{-C_{\mathtt{fin}}f(N(r))}.
	 \end{align*}
	 The last inequality follows from the definition of norm $|\cdot|_{r,s}.$
	\end{proof}
	We now proceed the estimate for the high mode.
	\begin{proposition}
		For $s>s_0$ and any real $w_0$ with $\Vert w_0\Vert_{s}<r\leq C_{\mathtt{thre}}$ in \eqref{Cau}, we have
		$$\Vert w^>(t)\Vert_s\leq\frac{1}{C_{\mathtt{sta}}}(\Vert w^>(0)\Vert_{s}+e^{-C_{\mathtt{fin}}f(N(r))}|t )|, \forall|t|\leq T_r.$$
	\end{proposition}
	\begin{proof}
		First, we denote by $\mathcal{L}(w^<):\Pi^>W_s\to\Pi^>W_s$ the family of linear operator such that $X_{Z_2}(w^<,w^>)=\mathcal{L}(w^<)w^>$, and denote $\mathcal{L}(t):=\mathcal{L}(w^<(t)).$ 
		
		Then for any $w^>\in\Pi^>W_s,$ we introduce the projectors defined as follows:
		$$\Pi_\alpha:\Pi^>W_s\to\Pi^>W_s\ ,  (w_{(j,\sigma)})_{(j,\sigma)}\mapsto (w_{(j,\sigma)} \chi_{\Omega_{\alpha}}(j))_{(j,\sigma)},$$ 
		where $\chi_{\Omega_{\alpha}}$ is indicator function on $\Omega_{\alpha}.$
		Then we can split $w$ as follows:
		$$\forall w\in\Pi^>W_s,\ w=\sum_{\alpha}w_\alpha,\ w_\alpha:=\Pi_{\alpha}w.$$
		Similarly, by the definition of case NR2, $\mathcal{L}(t)$ has a block-diagonal structure, namely it can be written as
		$$\mathcal{L}(t)=\sum_{\alpha}\mathcal{L}_{\alpha}(t),\ \mathcal{L}_\alpha(t)=\Pi_\alpha\mathcal{L}(t)\Pi_{\alpha}.$$
		For any block $\Omega_{\alpha}$ we define 
		$|\alpha|=\inf_{j\in\Omega_{\alpha}}|j|$. 
		Consider the normal form part of \eqref{high}, namely
		\begin{equation}\label{normeq}
			\partial_t w_{\alpha}(t)=\Lambda w_\alpha+\mathcal{L}_\alpha(t)z_{\alpha}(t).
		\end{equation}
		Since $\mathcal{L}_\alpha$ is Hamiltonian, we have
		$$\Vert w_\alpha(t)\Vert_{\ell^2}=\Vert w_\alpha(t_0)\Vert_{\ell^2}, \forall t,t_0\in [-T_r,T_r],$$
		therefore, $\forall |t|\leq T_r$		    
		\begin{align*}
			\Vert w(t)\Vert_s&=\sum_{\alpha}\sum_{j\in\Omega_{\alpha}}\e^{2sf(\langle j\rangle)}|w_{(j,\sigma)}(t)|^2 \\
			&\leq \sum_{\alpha}\sum_{j\in\Omega_{\alpha}}\e^{2sf(|\alpha|+C_1)}|w_{(j,\sigma)}(t)|^2 \\
			&\leq \sum_{\alpha}\e^{2s(f(|\alpha|)+C_f f(C_1))}\Vert w_{\alpha}(t)\Vert_{\ell_2}^2 \\
			&={C_{\mathtt{sta}}} \sum_{\alpha}\e^{2sf(|\alpha|)}\Vert w_{\alpha}(0)\Vert_{\ell_2}^2 \\
			&\leq {C_{\mathtt{sta}}}\sum_{\alpha}\sum_{j\in\Omega_{\alpha}}\e^{2sf(|j|)}|w_{(j,\sigma)}(0)|^2 \\
			&={C_{\mathtt{sta}}}\Vert w(0)\Vert_{s}.
		\end{align*}
		Hence, denoting by $\mathcal{W}(t,\tau)$ is the flow map of \eqref{normeq}, we have 
		$$\Vert\mathcal{W}(t,\tau)w_0\Vert_s\leq{C_{\mathtt{sta}}}\Vert w(0)\Vert_{s}.$$
		Now we can solve \eqref{high} as
		$$w^>(t)=\mathcal{W}(t,0)w_0+\int_{0}^{t}\mathcal{W}(t,\tau)\Pi^>X_{R_d}(w^<,w^>)d\tau.$$
		So we get 
		$$\Vert w^>(t)\Vert_s\leq{C_{\mathtt{sta}}}\Vert w_0\Vert_{s'}+{C_{\mathtt{sta}}}\e^{-C_{\mathtt{fin}}f(N)}|t|.$$

	\end{proof}
	
	We now combine the estimates for the low and high modes and apply a standard bootstrap argument to obtain the following result:
	\begin{theorem}[Main theorem]
		Consider Hamiltonian \eqref{Hamiltonian} with initial $u(0)=u_0$. Assume that $W_s$'s weight function $f$ satisfies assumption A.0, frequencies $\omega_{j}$ fulfill \textbf{Assumption 2} with $\beta>1$. Then for sufficiently large $s$, there exist a  threshold $\varepsilon_0>0, $ and constants $C_{\mathtt{sta}},C_{\mathtt{fin}}>0$ such that the following holds: if $u(0)$ is real and
		$$\varepsilon:=\Vert u(0)\Vert_{s}<\varepsilon_0,$$
		then $$\sup_{|t|\leq T_{\varepsilon}}\Vert u(t)\Vert_{s}<{C_{\mathtt{sta}}\varepsilon},$$
		where
		$$T_{\varepsilon}>\frac{\e^{C_{\mathtt{fin}}sf(N(\varepsilon))}}{C_{\mathtt{sta}}}.$$ 
		The explicit relation for $N(\varepsilon)$ is given in Theorem 1.
	\end{theorem}
	
	Based on this theorem, we use Lemma \ref{weight} to verify that $f(x)=x^{\theta}$ and $f(x)=\ln(x+\kappa)^q$ satisfy Assumption 1.  Then the main results of this paper, Theorem \ref{res1} and Theorem \ref{res2}, are proved.

	\section{Applications}\label{sec:app}
	In this section, we can specifically illustrate how our results improve previous ones, highlight the generalizations of this framework, and derive some new findings.
	
	\subsection{Schr\"odinger Equations with Convolution Potentials}
	We consider the classic Schr\"odinger equation \eqref{ex1} of the form
	\[\i\partial_t \psi=-\Delta\psi+V*\psi+p(|\psi|^2)\psi, x\in\mathbb{T}^{\mathsf{d}},\]
	where $V$ is a potential, $*$ denotes the convolution and the nonlinearity $p$ is in $C^\infty(\mathbb{R},\mathbb{R})$ and $p(0)=0$.
	Equation \eqref{ex1} is Hamiltonian with the Hamiltonian function
	$$H(\psi,\bar{\psi})=\int_{\mathbb{T}^{\mathsf{d}}}(|\nabla\psi|^2+\psi(V*\bar{\psi})+P(|\psi|^2))\d x,$$
	where $P$ is a primitive of $p$ in class $C^\infty(\mathbb{R},\mathbb{R})$ in a neighborhood of the origin and has a zero of order $2$ at the origin.
	
	When $V(x)=\frac{1}{|\mathbb{T}^\mathsf{d}|}\sum_{k\in\mathbb{Z}^{\mathsf{d}}}V_ke^{ikx},$ we consider the space:
	$$\mathcal{V}=\{V\mid V_k |k|^n\in[-\frac{1}{2},\frac{1}{2}]\},$$
	and endow with product probability measure. We present the following results:
	\begin{theorem}[Gevrey class case]\label{ex1thm1}
		There exists a zero measure set $\mathcal{V}^{\mathtt{res}}\subset\mathcal{V}$ such that $\forall V\in\mathcal{V}\setminus\mathcal{V}^{\mathtt{res}}, s>S_{\mathtt{fin}},\varepsilon<\varepsilon_0,$ if initial data $\psi_0$ of \eqref{ex1} satisfies $\Vert \psi_0\Vert_{s,\theta}^G=\varepsilon,$ then the solution of \eqref{ex1} satisfies
		$$\Vert\psi(t)\Vert_{s,\theta}^G\leq C_{\mathtt{sta}}\varepsilon ,\ \forall |t|\leq \frac{1}{C_{\mathtt{sta}}}e^{C_{\mathtt{fin}}\frac{|\ln \varepsilon|^2}{\ln|\ln \varepsilon|}}.$$
	\end{theorem}
	\begin{theorem}[Logarithmic Ultra-differential case]\label{ex1thm2}
		There exists a zero measure set $\mathcal{V}^{\mathtt{res}}\subset\mathcal{V}$ such that $\forall V\in\mathcal{V}\setminus\mathcal{V}^{\mathtt{res}}, s>S_{\mathtt{fin}},\varepsilon<\varepsilon_0,$ if initial data $\psi_0$ of \eqref{ex1} satisfies $\Vert \psi_0\Vert_{s,q}^U=\varepsilon,$ then the solution of \eqref{ex1} satisfies
		$$\Vert\psi(t)\Vert_{s,q}^U\leq C_{\mathtt{sta}}\varepsilon ,\ \forall |t|\leq \frac{1}{C_{\mathtt{sta}}}e^{C_{\mathtt{fin}}|\ln\varepsilon|^{1+a}},$$
		where $a\leq\frac{q-1}{q+1}.$
	\end{theorem}
	It remains to verify that the frequencies in the Hamiltonian of equation \eqref{ex1} satisfy assumptions A.1, A.2, A.3 and prove that the set of frequencies violating these assumptions has zero measure.
	
	To fit our scheme, we introduce the Fourier coefficients
	$$\psi(x)=\frac{1}{\sqrt{|\mathbb{T}^{\mathsf{d}}|}}\sum_{j\in\mathbb{Z}^\mathsf{d}}u_{j,+}\e^{ijx},\bar{\psi}(x)=\frac{1}{\sqrt{|\mathbb{T}^{\mathsf{d}}|}}\sum_{j\in\mathbb{Z}^\mathsf{d}}u_{j,-}\e^{-ijx}.$$
	In these variables, equation \eqref{ex1} takes the form $H=H_0+P,$ where $ H_0$ has frequencies
	$$\omega_{j}:=|j|^2+V_j.$$ 
	Then the frequencies belong to the set
	$$D_n=\{\omega\mid\sup_{j\in\mathbb{Z}^{\mathsf{d}}}|\omega-|j|^2||j|^n<\frac{1}{2}\}.$$
	
    Obviously, 
    $$\frac{|j|^2}{2}\leq|j|^2-\frac{1}{2}\leq|\omega_{j}|\leq|j|^2+\frac{1}{2}\leq\frac{3}{2}|j|^2, \forall j\neq0,$$
    so assumption A.1 is satisfied with $\beta=2, C_0=2.$
    Besides, $\forall j\neq k,$
    \begin{align*}
    	|\omega_j-\omega_k|\geq|j|^2-|k|^2-1\geq(|j|-|k|(|j|+|k|))-1\geq\frac{|j|+|k|}{2}.
    \end{align*}
	Then we take every $\Omega_\alpha$ as a spherical shell of thickness $C_1$, and $C_2=\frac{1}{2},\delta=1$ to satisfy assumption A.3.
	
	In current studies of the Schrödinger equation with external parameters, Bourgain's non-resonance condition for convolutions has been extensively employed. We can briefly illustrate that our non-resonance condition encompasses this type of non-resonance assumption.
	
	In \cite{B05}, Bourgain established a classical measure estimate for the resonant set associated with a convolution potential. This work was extended in \cite{BMP20} to the following more general framework
	$$D_{\gamma,n}^{\mu_1,\mu_2}=\{\omega\in D_n\mid |\omega\cdot\ell|>\gamma\prod_{m\in\mathbb{Z}^{\mathsf{d}}}\frac{1}{(1+\ell_m^{\mu_1}\langle m\rangle^{\mu_2+n})},\forall\ell\in\mathbb{Z}^{\mathbb{Z}^{\mathsf{d}}}\},$$
	where $\mu_1,\mu_2>1.$ Notice that when supportive index $m$ for $\ell$ satisfying $\langle m\rangle<N-1$, we have 
	\begin{align*}
		\prod_{m\in\mathbb{Z}^{\mathsf{d}}}{(1+\ell_m^{\mu_1}\langle m\rangle^{\mu_2+n})}&\leq\prod_{m\in\mathbb{Z}^\mathsf{d}}(1+l_m\langle m\rangle)^{\mu_1+\mu_2+n}\\
		&\leq\prod_{m\in\mathbb{Z}^\mathsf{d}}(1+\langle m\rangle)^{l_m(\mu_1+\mu_2+n)}\\
		&\leq\prod_{m\in\mathbb{Z}^\mathsf{d}}(1+\langle m\rangle)^{l_m(\mu_1+\mu_2+n)}\\
		&\leq N^{(\mu_1+\mu_2+n)|\ell|_1}.
	\end{align*}
	So $\omega\in D_{\gamma,n}^{\mu_1,\mu_2}$ actually also holds:
	$$|\omega\cdot\ell|\geq\frac{\gamma}{N^{\tau d}},$$
	where $d=|\ell|_1$ is just the $1$- norm of $\ell,$ and $\tau=n+\mu_1+\mu_2.$
	Namely, assumption A.2 is satisfied by $\tau=\mu_1+\mu_2+n,p=1.$

	\subsection{Fractional Schr\"odinger equation}
	In this section, we present a case of $p\neq1$ arising from the weakening of the non-resonance condition, such as only one parameter is used to adjust non-resonance.
	
	Now we study the following fractional Schr\"odinger equation \eqref{ex2}
	\[\i\partial_t\psi=(\Delta+m)^\eta\psi+p(|\psi|^2)\psi, \]
	where $\eta>\frac{1}{2}$ satisfies the assumption A.1, and $p(x)$ is the same as in the previous subsection. Then \eqref{ex2} can be viewed as Hamiltonian system with Hamiltonian function
	$$H(\psi,\bar{\psi})=\int_{x\in\mathbb{T}^\mathsf{d}}\bar{\psi}(\Delta+m)^\eta \psi+P(|\psi|^2)\d x,$$
	where $P$ is a primitive of $p$ in class $C^{\infty}(\mathbb{R},\mathbb{R})$ in a neighborhood of the origin and has a zero of order 2 at the origin.
	When we use the Fourier expansion
	$$u_{\sigma}(x):=\frac{1}{\sqrt{|\mathbb{T}^\mathsf{d}|}}\sum_{j\in\mathbb{Z}^{\mathsf{d}}}u_{(j,\sigma)}\e^{ijx},$$
    equation \eqref{ex2} takes the form $H=H_0+P,$ with frequency
	$$\omega_j=(|j|^2+m)^{\eta}.$$
	We use the parameter $m$ to adjust the non-resonance. The results are as follows:
	\begin{theorem}[Gevrey class case]\label{ex2thm1}
		For any interval $[M_1,M_2]$, there exists a zero measure set $\mathcal{M}\subset[M_1,M_2]$ such that $\forall m\in[M_1,M_2]\setminus\mathcal{M}, s>S_{\mathtt{fin}},\varepsilon<\varepsilon_0,$ if the initial data $\psi_0$ of \eqref{ex2} satisfies $\Vert \psi_0\Vert_{s,\theta}^G=\varepsilon,$ then the solution of \eqref{ex2} satisfies
		$$\Vert\psi(t)\Vert_{s,\theta}^G\leq C_{\mathtt{sta}}\varepsilon ,\ \forall |t|\leq \frac{1}{C_{\mathtt{sta}}}e^{C_{\mathtt{fin}}\frac{|\ln \varepsilon|^2}{\ln|\ln \varepsilon|}}.$$
	\end{theorem}
	\begin{theorem}[Logarithmic Ultra-differential case]\label{ex2thm2}
		For any interval $[M_1,M_2]$, there exists a zero measure set $\mathcal{M}\subset[M_1,M_2]$ such that $\forall m\in[M_1,M_2]\setminus\mathcal{M}, s>S_{\mathtt{fin}},\varepsilon<\varepsilon_0,$ if the initial data $\psi_0$ of \eqref{ex2} satisfies $\Vert \psi_0\Vert_{s,q}^U=\varepsilon,$ then the solution of \eqref{ex2} satisfies
		$$\Vert\psi(t)\Vert_{s,q}^U\leq C_{\mathtt{sta}}\varepsilon ,\ \forall |t|\leq \frac{1}{C_{\mathtt{sta}}}e^{C_{\mathtt{fin}}|\ln\varepsilon|^{1+a}},$$
		where $a\leq\frac{q-1}{3q+1}.$
	\end{theorem}
	It's easy to verify when $\beta=2\eta,\delta=2\eta-1$, \eqref{ex2} satisfies A.1, A.3. Thus, we just need to construct the resonant set $\mathcal{M},$ and estimate the measure with a standard process.
	
	\begin{lemma}
		For $1\leq k\leq d$ and $|j_1|<\dots<|j_k|<N$, consider the determinant
		$$
		 	D:=\begin{vmatrix}\omega_{j_1}&\omega_{j_2}&\dots&\omega_{j_k}\\
		 	\frac{d\omega_{j_1}}{dm}&\frac{d\omega_{j_2}}{dm}&\dots&\frac{d\omega_{j_k}}{dm}\\
		 	\dots&\dots&\dots&\dots\\
		 	\frac{d^{k-1}\omega_{j_1}}{dm^{k-1}}&\frac{d^{k-1}\omega_{j_2}}{dm^{k-1}}&\dots&\frac{d^{k-1}\omega_{j_k}}{dm^{k-1}}\end{vmatrix}. 
		 $$
	     We have $|D|\geq\frac{C_\eta}{N^{2d^2}}.$
	\end{lemma}
	\begin{proof}
		We can calculate 
		$$\frac{\d^l \omega_{j}}{\d m^l}=(|j|^2+m)^{\eta-l}\prod_{n=0}^{l-1}(\eta-n),$$
		so 
		$$
		D=\prod_{l=1}^{k}\omega_{j_l}\prod_{l=0}^{k-1}(\eta-l)^{k-l}
		\begin{vmatrix}
			1&  1&  \dots&  1& \\
			x_1&  x_2&  \dots&  x_k& \\
			\dots&  \dots&  \dots&  \dots& \\
			x_1^{k-1}& x_2^{k-1} &  \dots& x_k^{k-1} &
		\end{vmatrix},
		$$
		where $x_l=\frac{1}{|j_l|^2+m}.$
		The last determinant is a Vandermonde determinant and can be expressed as 
		\begin{align*}
			\prod_{1\leq r\leq s\leq k}(x_{j_r}-x_{j_s})&=\prod_{1\leq r\leq s\leq k}\frac{|j_r|^2-|j_s|^2}{(|j_r|^2+m)(|j_s|^2+m)}\\
			&=\prod_{1\leq r\leq s\leq k}(|j_r|^2-|j_s|^2)(\prod_{1\leq l\leq k}\frac{1}{|j_l|^2+m})^{k-1}.
		\end{align*}
		Thus we have 
		\begin{align*}
			|D|\geq C_{\eta}(\prod_{1\leq l\leq k}\frac{1}{2N^2})^{k-1}\geq\frac{C_{\eta}}{N^{2k^2}},
		\end{align*}
		and the conclusion follows for $k\leq d.$
	\end{proof}
	Using Lemmas \ref{det} and \ref{measure}, we derive the following measure estimate for the non-resonant set:
	$$\mathcal{M}_\gamma=\{m\in[M_1,M_2]\mid\sum_{l=1}^{d}\sigma_l\omega_{j_l}\leq\frac{\gamma}{N^{4d^3}}, \exists \{j_1,...,j_d\}, |j_l|<N\}.$$
	
	\begin{proposition}
     $$|\mathcal{M}_\gamma|\leq \gamma.$$ 
	\end{proposition} 
	\begin{proof}
	 By Lemma \ref{det}, for any $\mathcal{J}=\{j_1,...,j_d\}$ satisfying $|j_l|<N,$ we can get an index $(i)$ such that
		$$|\sum_{l=1}^{d}\frac{\d^{(i)}\omega_{j_l}(m)}{\d m^{(i)}}|\geq\frac{C_{\eta}d}{N^{2d^2+2}}.$$
	We fix $\mathcal{J}$ to define $$\mathcal{M}_{\mathcal{J,\gamma}}:=\{m\in[M_1,M_2]\mid\sum_{l=1}^{d}\sigma_l\omega_{j_l}(m)\leq\frac{\gamma}{N^{4d^3}}\}.$$
	 Then by Lemma \ref{measure}, we have
	 \begin{align*}
	 	|\mathcal{M}_{\mathcal{J},\gamma}|&\leq\left(\frac{\gamma}{N^{4d^3}}\right)^{\frac{1}{(i)}}\frac{N^{2d^2+2}}{C_\eta}\\
	 	&\leq(\frac{\gamma^{\frac{1}{(i)}}}{N^{4d^2}})\frac{N^{2d^2+2}}{C_\eta}\leq \frac{C_\eta\gamma}{N^{d^2}}.
	 	\end{align*}
	 Thus
	 \begin{align*}
	 	|\mathcal{M}_\gamma|&\leq\sum_{\mathcal{J},|j_l|<N}|\mathcal{M}_{\mathcal{J},\gamma}|\\
	 	&\leq\sum_{\mathcal{J},|j_l|<N}\frac{C_\eta\gamma}{N^{d^2}}\leq\frac{C_\eta\gamma(2N)^d}{N^{d^2}}\leq\gamma.
	 \end{align*} 
	\end{proof}
	Eventually, we make
	$$\mathcal{M}:=\bigcap_{\gamma>0}\mathcal{M}_\gamma$$ 
   as the resonant set $\mathcal{M}$ in Theorems in this subsection. Namely, for $m\in[M_1,M_2]\setminus\mathcal{M}$, the frequencies $\omega_j$ satisfy the assumption A.2.

	\subsection{Beam equation}
	In this section, we present another case of weak non-resonance, namely use metric $g$ to adjust non-resonance. The detail setting for metric on $\mathbb{T}^\mathsf{d}$ can be seen in section 5 in \cite{BFM24}, and we insert it for the sake of completeness.
	
	 Let $e_1,\dots,e_{\mathtt{d}}$ be a basis of $\mathbb{R}^{\mathtt{d}}$ and let
	 $$\Gamma:=\{x\in\mathbb{R}^{\mathsf{d}}:x=\sum_{j=1}^{\mathsf{d}}2\pi n_je_j,n_j\in\mathbb{Z}\}$$
	 be a maximal dimensional lattice. We denote $\mathbb{T}^\mathsf{d}_{\Gamma}:=\mathbb{R}^\mathsf{d}/\Gamma$.
	 
	To fit our scheme, it is convenient to introduce in $\mathbb{T}^\mathsf{d}_{\Gamma}$ the basis given by $e_1,...,e_d$, so
	that the functions turn out to be defined on the standard torus $\mathbb{T}^\mathsf{d}$ but endowed by the metric $\mathsf{g}_{ij}=e_j\cdot e_i$. In particular, the Laplacian operator in this metric is expressed as
	$$\Delta_{g}=\sum_{i,j=1}^{\mathsf{d}}g_{ij}\partial_{x_i}\partial_{x_j},$$ 
	where $g_{i,j}$ is the inverse of matrix $\mathsf{g}_{i,j}.$ The positive definite symmetric quadratic form $g(k,k)$ is defined by
	$$g(k,k):=\sum_{i,j=1}^{\mathsf{d}}g_{ij}k_ik_j,\forall k\in\mathbb{Z}^\mathsf{d},$$ 
	and $\Vert g\Vert_2^2:=\sum_{i,j}|g_{ij}|^2.$ Then we denote $\tau^*=\frac{\mathsf{d(\mathsf{d+1})}}{2}$ for the open set
	$$\mathcal{G}_0:=\{(g_{ij})_{i\leq j}\in\mathbb{R}^{\tau^*}\mid\inf_{x\neq0}\frac{g(x,x)}{|x|^2}>0\}.$$
	Define the set of admissible metrics as follows
	$$\mathcal{G}:=\bigcup_{\Gamma>0}\mathcal{G}_{\Gamma},$$
	where
	$$\mathcal{G}_\Gamma:=\{g\in\mathcal{G}_0\mid|\sum_{i\leq j}g_{ij}\ell_{ij}|\geq\frac{\Gamma}{(\sum_{i\leq j}|\ell_{ij}|)^{\tau^*}},\forall \ell\in\mathbb{R}^{\tau^*}\setminus\{0\}\}.$$
	Besides, we set 
	$$\mathcal{G}(\zeta_1,\zeta_2):=\{g\in\mathcal{G}\mid\zeta_1\leq\Vert g\Vert_2\leq\zeta_2\},$$
	$$\mathcal{G}_0(\zeta_1,\zeta_2):=\{g\in\mathcal{G}_0\mid\zeta_1\leq\Vert g\Vert_2\leq\zeta_2\}.$$

	Now we study the following beam equation \eqref{ex3} 
	\[\psi_{tt}+\Delta_g^2\psi+m\psi=-\frac{\partial p}{\partial \psi}+\sum_{l=1}^{L}\partial_{x_l}\frac{\partial p}{\partial(\partial_l\psi)} \]
	
	with $p(\psi,\partial_{x_1},\dots,\partial_{x_L})$ a function of class $C^\infty(\mathbb{R}^{\mathsf{d}+1},\mathbb{R})$ in a neighborhood of the origin and a zero of order $2$ at the origin. Introducing the variable $\phi=\dot{\psi}=\psi_{t},$ \eqref{ex3} can be seen as an Hamiltonian system in the variables $(\psi,\phi)$ with Hamiltonian function
	$$H(\psi,\phi):=\int_{\mathbb{T}^\mathsf{d}}\left(\frac{\phi^2}{2}+\frac{\psi(\Delta_g^2+m)\psi}{2}+p(\psi,\partial_{x_1},...,\partial_{x_L})\right)\d x.$$
	Then we can introduce new variables
	$$u_{\sigma}(x):=\frac{1}{\sqrt{2}}\left((\Delta_g^2+m)^{\frac{1}{4}}\phi+\sigma\i(\Delta_g^2+m)^{-\frac{1}{4}}\psi \right),$$
	and consider the Fourier series
	$$u_{\sigma}(x):=\frac{1}{\sqrt{|\mathbb{T}^\mathsf{d}|_g}}\sum_{j\in\mathbb{Z}^{\mathsf{d}}}u_{(j,\sigma)}\e^{ijx}.$$
	In these variables the beam equation \eqref{ex3} takes the form $H=H_0+P$, where $P$ is obtained by substituting $p$ term of the Hamiltonian and in $\mathcal{P},$ and $H_0$ has frequencies
	$$\omega_j=\sqrt{|j|_g^4+m}.$$
	We will illustrate that the metric $g$ contribute to the non-resonance condition. Our results for equation \eqref{ex3} are as follows:
	\begin{theorem}[Gevrey class case]\label{ex3thm1}
		For $0<\zeta_1<\zeta_2$, there exists a zero measure set $\mathcal{G}^{\mathtt{res}}\subset\mathcal{G}_0(\zeta_1,\zeta_2)$ such that $\forall g\in\mathcal{G}_0(\zeta_1,\zeta_2)\setminus\mathcal{G}^{\mathtt{res}}, s>S_{\mathtt{fin}},\varepsilon<\varepsilon_0,$ if initial data $\psi_0$ of \eqref{ex3} satisfies $\Vert \psi_0\Vert_{s,\theta}^G=\varepsilon,$ then the solution of \eqref{ex3} satisfies
		$$\Vert\psi(t)\Vert_{s,\theta}^G\leq C_{\mathtt{sta}}\varepsilon ,\ \forall |t|\leq \frac{1}{C_{\mathtt{sta}}}e^{C_{\mathtt{fin}}\frac{|\ln \varepsilon|^2}{\ln|\ln \varepsilon|}}.$$
	\end{theorem}
	\begin{theorem}[Logarithmic Ultra-differential case]\label{ex3thm2}
		For $0<\zeta_1<\zeta_2$, there exists a zero measure set $\mathcal{G}^{\mathtt{res}}\subset\mathcal{G}_0(\zeta_1,\zeta_2)$ such that $\forall g\in\mathcal{G}_0(\zeta_1,\zeta_2)\setminus\mathcal{G}^{\mathtt{res}},  s>S_{\mathtt{fin}},\varepsilon<\varepsilon_0,$ if initial data $\psi_0$ of \eqref{ex3} satisfies $\Vert \psi_0\Vert_{s,q}^U=\varepsilon,$ then the solution of \eqref{ex3} satisfies
		$$\Vert\psi(t)\Vert_{s,q}^U\leq C_{\mathtt{sta}}\varepsilon ,\ \forall |t|\leq \frac{1}{C_{\mathtt{sta}}}e^{C_{\mathtt{fin}}|\ln\varepsilon|^{1+a}},$$
		where $a\leq\frac{q-1}{3q+1}.$
	\end{theorem}
	Like previous subsection, we can choose $\beta=2,\delta=1$ to satisfy assumptions A.1, A.3. We mainly construct the resonant set that violate assumption A.2 and make measure estimate here. 
	
	We can firstly get the following result:
	\begin{proposition}
		For a fix $\bar{g}\in\mathcal{G}_0,$ there exists a zero measure set $\mathcal{Z}^{\mathtt{res}}_{\bar{g}}\subset[\zeta_1,\zeta_2]$, such that when $\zeta\in\mathcal{Z}^{\mathtt{nr}}_{\bar{g}}:=[\zeta_1,\zeta_2]\setminus\mathcal{Z}^{\mathtt{res}}_{\bar{g}},$
		the frequency $$\omega_{j}=\sqrt{|j|_g^4+m}$$
		satisfies assumption A.2, where metric $g=\zeta\bar{g}.$ 
	\end{proposition}
	\begin{proof}
		We can make substitute for
		$$\omega_j(\zeta)=\zeta^2\Omega_j,\ \Omega_j:=\sqrt{|j|_{\bar{g}}^4+\frac{m}{\zeta^4}}, $$
		and use
		$$(\zeta_1,\zeta_2)\to(\xi_1,\xi_2):=(\frac{m}{\zeta_2},\frac{m}{\zeta_1}),\zeta\mapsto\xi:=\frac{m}{\zeta^4}.$$
		Notice that this map is an analytic diffeomorphism, we just need to prove that there exists a zero measure set $\mathcal{K}^{\mathtt{res}}$ such that frequencies
		$$\Omega_j=\sqrt{|j|_g^4+\xi}$$
		satisfy assumption A.2 for $\xi\in[\xi_1,\xi_2]\setminus\mathcal{K}^{\mathtt{res}}.$ Then we can get
		$$
		D=\prod_{l=1}^{k}\Omega_{j_l}\prod_{l=0}^{k-1}(\frac{1}{2}-l)^{k-l}
		\begin{vmatrix}
			1&  1&  \dots&  1& \\
			x_1&  x_2&  \dots&  x_k& \\
			\dots&  \dots&  \dots&  \dots& \\
			x_1^{k-1}& x_2^{k-1} &  \dots& x_k^{k-1} &
		\end{vmatrix},
		$$
		for
		$$
		D:=\begin{vmatrix}\Omega_{j_1}&\Omega_{j_2}&...&\Omega_{j_k}\\
			\frac{d\Omega_{j_1}}{d\xi}&\frac{d\Omega_{j_2}}{d\xi}&...&\frac{d\Omega_{j_k}}{d\xi}\\
			\dots&\dots&\dots&\dots\\
			\frac{d^{k-1}\Omega_{j_1}}{d\xi^{k-1}}&\frac{d^{k-1}\Omega_{j_2}}{d\xi^{k-1}}&\dots&\frac{d^{k-1}\Omega_{j_k}}{d\xi^{k-1}}\end{vmatrix}, 
		$$
		where $x_l=\frac{1}{|j_l|_g^4+\xi}.$ To estimate the lower bound of Vandermonde determinant, we need to use the definition of $\mathcal{G}$ to get separation between $|j_r|_g^4$ and $|j_s|_{\bar{g}}^4$. If we denote $j_r=R,j_s=S$ in this proof, we have 
		\begin{align*}
		\left||R|_{\bar{g}}^2-|S|_{\bar{g}}^2\right|&\geq|\sum_{i,j}\bar{g}_{i,j}R_iR_j-S_iS_j|\\
		&\geq\frac{\Gamma}{(\sum_{i,j}R_iR_j-S_iS_j)^{\tau^*}}\\
		&\geq\frac{\Gamma}{(\sum_{i,j}|R_i||R_j|+|S_i||S_j|)^{\tau^*}}\\
		&\geq\frac{\Gamma}{(|R|^2+|S|^2)^\tau}\geq\frac{\Gamma}{(2N)^{2\tau^*}}.
		\end{align*} 
		Then the Vandermonde determinant implies
		\begin{align*}
			\prod_{1\leq r<s\leq k}|x_r-x_s|&=\prod_{1\leq r<s\leq k}|\frac{1}{|j_r|_g^4+\xi}-\frac{1}{|j_s|_g^4+\xi}|\\
			&\geq\prod_{1\leq r<s\leq k}\frac{(|j_r|_g^2+|j_s|_g^2)||R|_g^2-|S|_g^2|}{(|j_r|_g^4+\xi)(|j_s|_g^4+\xi)}\\
			&\geq\prod_{1\leq r<s\leq k}\frac{||R|_g^2-|S|_g^2|}{(|j_r|_g^4+\xi)(|j_s|_g^4+\xi)}\\
			&\geq\frac{\Gamma^{d^2}}{N^{4\tau^*d^2}}.
		\end{align*}
		We then define a non-resonant set for a fixed multi-index $\mathcal{J}=\{j_1,\dots,j_d\}$:
		$$\mathcal{K}_{\mathcal{J},\gamma}=\{\xi\in[\xi_1,\xi_2]\mid|\sum_{l=1}^{d}\sigma_l\Omega_{j_l}(\xi)|\geq\frac{\gamma}{N^{4(\tau^*+1)d^3}}\},$$ 
		where we take $\gamma=\Gamma^{d^3+d}$ and we can make the measure estimate as the previous subsection
		\begin{align*}
			|\mathcal{K}_{\mathcal{J},\gamma}|\leq2\frac{N^{4\tau^*d^2}}{\Gamma^{d^2}}(\frac{\gamma}{N^{4(\tau^*+1)d^3}})^{\frac{1}{d}}\leq\frac{2\gamma}{N^{4d^2}}.
		\end{align*}
		Then we define 
		$$\mathcal{K}_{\gamma}:=\{\xi\in[\xi_1,\xi_2]\mid|\sum_{l=1}^{d}\sigma_l\Omega_{j_l}(\xi)|\geq\frac{\gamma}{N^{4d^5}},\exists\{j_1,\dots,j_d\},|j_l|<N\},$$
		and
		$$\mathcal{K}^{\mathtt{res}}=\bigcap_{\gamma>0}\mathcal{K}_\gamma$$
		is a zero measure set as we desired.
	\end{proof}
	We now proceed to discuss the non-resonant property of the metric set.
	Let $\partial B_r:=\{\Vert g\Vert_2=r\}$ denote a sphere in the metric space, let $\mu_r$ represent the $\tau^*-1$ dimensional measure on $\partial B_r$, and let $\lambda$ be the $\tau^*$ dimension measure in metric space. When $A\subset[\zeta_1,\zeta_2]$, we define $\bar{g}A=\{g\mid g=\bar{g}a,a\in A\}.$ We will prove that the non-resonant set
	$$\mathcal{G}^{\mathtt{nr}}(\zeta_1,\zeta_2):=\cup_{\bar{g}\in\partial B_1\cap\mathcal{G}}\bar{g}\mathcal{Z}_{\bar{g}}^{\mathtt{nr}}$$
	has full measure in $\mathcal{G}_0(\zeta_1,\zeta_2)$.
	\begin{proposition}
		$\mathcal{G}^{\mathtt{res}}=\mathcal{G}_0\setminus\mathcal{G}^{\mathtt{nr}}(\zeta_1,\zeta_2)$ is the zero measure set that make frequency violate assumption A.2.
	\end{proposition}
	\begin{proof}
		From the setting of $\mathcal{G}$ we can know that \begin{align*}
			\lambda(\mathcal{G}(\zeta_1,\zeta_2))&=\lambda(\mathcal{G}_0(\zeta_1,\zeta_2)),\\
		\int_{\zeta_1}^{\zeta_2}\mu_{\zeta}(\mathcal{G}\cap\partial B_{\zeta})\d \zeta	&=\int_{\zeta_1}^{\zeta_2}\mu_{\zeta}(\mathcal{G}_0\cap\partial B_{\zeta})\d \zeta.
		\end{align*}
    Use the scaling properties, we have
	\begin{align*}
		\int_{\zeta_1}^{\zeta_2}\zeta^{\tau^*-1}\mu_{1}(\mathcal{G}\cap\partial B_{1})\d \zeta&=\int_{\zeta_1}^{\zeta_2}\zeta^{\tau^*-1}\mu_{1}(\mathcal{G}_0\cap\partial B_{1})\d \zeta,\\
		\frac{\zeta_2^{\tau^*}-\zeta_1^{\tau^*}}{\tau^*}\mu_{1}(\mathcal{G}\cap\partial B_{1})&=	\frac{\zeta_2^{\tau^*}-\zeta_1^{\tau^*}}{\tau^*}\mu_{1}(\mathcal{G}_0\cap\partial B_{1}),\\
	\mu_{1}(\mathcal{G}\cap\partial B_{1})	&=\mu_{1}(\mathcal{G}_0\cap\partial B_{1}).
	\end{align*}
	Then by Fubini's theorem, we have
	\begin{align*}
		\lambda(\mathcal{G}^{\mathtt{nr}}(\zeta_1,\zeta_2))&=\int_{\partial B_1\cap\mathcal{G}}|\mathcal{Z}_{\bar{g}}^{\mathtt{nr}}|\d\mu_{1}(\bar{g})\\
		&=(\zeta_2-\zeta_1)\int_{\partial B_1\cap\mathcal{G}}\d\mu_{1}(\bar{g})\\
		&=(\zeta_2-\zeta_1)\mu_{1}(\mathcal{G}\cap\partial B_{1})\\
		&=(\zeta_2-\zeta_1)\mu_{1}(\mathcal{G}_0\cap\partial B_{1})\\
		&=\lambda(\mathcal{G}_0(\zeta_1,\zeta_2)).
	\end{align*} 
	And clearly $\mathcal{G}_0\setminus\mathcal{G}^{\mathtt{nr}}(\zeta_1,\zeta_2)\subset\mathcal{G}_0(\zeta_1,\zeta_2)$, so $\lambda(\mathcal{G}^{\mathtt{res}})\leq\lambda(\mathcal{G}_0(\zeta_1,\zeta_2))-\lambda(\mathcal{G}^{\mathtt{nr}}(\zeta_1,\zeta_2))=0,$
	which comes to the conclusion.
	\end{proof}
	From the above proposition, we has given the full measure set $\mathcal{G}^{\mathtt{nr}}(\zeta_1,\zeta_2)$ ensuring the validity of Assumption A.2.

	\appendix
	\renewcommand{\thesection}{\Alph{section}}
	\section*{Appendices}
	\section{Constants}
	\begin{align*}
		C_{\mathtt{sep}}&=C_0^{\frac{2}{\beta}},\\
	C_{\mathtt{deno}}&=(\frac{C_0}{C_2}+C_0^2)^\beta,\\
	C_{\mathtt{exp}}&=\tau(1+\frac{\beta}{\delta})+1,\\
	C_{\mathtt{estP}}&=\frac{64\e^2C_P^2}{\gamma},\\
	C_{\mathtt{thre}}&=\max\{\frac{32C_P\e}{\gamma},2,\frac{24\e^2}{\gamma},\frac{C_{\mathtt{estP}}16\e}{\gamma},\e^{2sC_ff(C_1)}\},\\
	C_{\mathtt{rema}}&=\max\{48\e(\e^{\frac{1}{16\e}}-1),C_{\mathtt{estP}},\e C_P,C_{\mathtt{deno}}\},\\
	C_{\mathtt{fin}}&=2^{p+2}C_{\mathtt{exp}},\\
	D_{\mathtt{fin}}&=\max\{4C_{\mathtt{rema}},\frac{32\e^2}{\gamma}\},\\
	S_{\mathtt{fin}}&=s_0+C_{\mathtt{fin}},\\
	C_{\mathtt{sta}}&=\e^{-2sC_ff(C_1)}.\\
	\end{align*}

	\section{Technical Lemmas }
	\begin{lemma}[Norm estimate for $P$]\label{norm}
		When $s>s_0$, for any $P\in\mathcal{P}_d,d\geq3,$ we have $$|P|_{r,s}\leq C_Pr^{d-2},$$ where $s_0$ satisfies $\sum_{J\in\mathcal{Z}}e^{(2C_f-2)s_0f(|J|)}<\frac{1}{3}.$ 
	\end{lemma}
	\begin{proof}
		Let $P=\sum_{\mathcal{J}\in\mathcal{I}_d}P_{\mathcal{J}}u_{J_1}...u_{J_d},$ denote multi-index $\{J_1,...J_{k-1},(j,-1),J_{k+1},...,J_d\}$ by $\hat{J}_{k,j}$, and denote $\{J_1,...J_{k-1},J_{k+1},...,J_d\}$ by $\hat{J}_k,$ then
		\begin{align*}
			(X_P)_{j,+1}&=-\i\sum_{k=1}^{d}\sum_{\hat{J}_{k,j}\in \mathcal{I}_d}P_{\hat{J}_{k,j}}u_{\hat{J}_k},\\
			|(X_P)_{j,+1}|&\leq C_P\sum_{k=1}^{d}\sum_{\hat{J}_{k,j}\in \mathcal{I}_d}|u_{\hat{J}_k}|,\\
			|(X_P)_{j,+1}|e^{sf(\langle j\rangle)}&\leq C_P\sum_{k=1}^{d}\sum_{\hat{J}_{k,j}\in \mathcal{I}_d}|u_{\hat{J}_k}|e^{sf(|j|)}.
		\end{align*}
		Notice that $|j|=|\mathcal{M}(J_1,...,J_{k-1},J_{k+1},J_d)|,$ from $\mathcal{M}(\hat{J}_{k,j})=0.$   When $d\geq3$,  we have
		\begin{align*}
			sf(\langle j\rangle)\leq sf(\sum_{l\neq k}\langle J_l\rangle)\leq sf(\langle J_m\rangle)+sC_f(\sum_{l\neq m,k}\langle J_l\rangle).
		\end{align*} 
		We omit a technical discussion here. 
		Then 
		\begin{align*}
			|(X_P)_{j,+1}|e^{sf(\langle j\rangle)}&\leq C_P\sum_{k=1}^{d}\sum_{\hat{J}_{k,j}\in \mathcal{I}_d}
			e^{(1-C_f)sf(|J_m|)}\prod_{J\in\hat{J}_k}|u_{J}|e^{sC_f(\langle J\rangle)},\\
			\sum_{j\in\mathbb{Z}}|(X_P)_{j,+1}e^{sf(\langle j\rangle)}|^2&\leq C_P^2
			\sum_{j\in \mathbb{Z}} 
			\left(\sum_{k=1}^{d}\sum_{\hat{J}_{k,j}\in \mathcal{I}_d}
			e^{(1-C_f)sf(\langle J_m\rangle)}\prod_{J\in\hat{J}_k}|u_{J}|e^{sC_f(\langle J\rangle)}\right)^2\\
			&\leq d^2C_P^2(\sum_{J_m}|u_m|e^{sf(\langle J_m\rangle)})\prod_{J\neq J_m,J\in\hat{J}_k}\left(\sum_{J}|u_J|e^{sC_ff(\langle J\rangle )} \right)^2\\
			&\leq d^2C_P^2(\sum_{J_m}|u_m|e^{2sf(|J_m|)})\\
			&\prod_{J\neq J_m,J\in\hat{J}_k}(\sum_{J}|u_J|^2e^{2sf(|J|)})(\sum_{J}e^{(2C_f-2)s_0f(|J|)})^{d-2}\\
			&\leq \frac{d^2}{9^{d-2}}C_P^2\Vert u\Vert_s^{2d-2}.
		\end{align*}
		So $\Vert X_P\Vert_s\leq C_P\Vert u\Vert_s^{d-1}$ comes to the conclusion $|P|_{r,s}\leq C_Pr^{d-2}$.
	\end{proof}
	
	\begin{lemma}[Cutting lemma]\label{cut}
		For $s>s_0,$
		if monomials $P\in\mathcal{P}_d, d\geq3$ have at least 3 degree zero at $u^>=0$, then we have
		$$|P|_{r,s} \leq C_{P}\frac{(2r)^{d-2}}{e^{(s-s_0)f(N)}}.$$ 
		
	\end{lemma}
	\begin{proof}
		Because we can firstly make binomial expansion $P(u)=\sum_{l=3}^{d}P'_{J,l}(u^>)^l(u^<)^{d-l}$, and like the proof of Lemma 3.8 in \cite{BFM24} , we have
		\begin{align*}
			\Vert(X_P)(u^>,u^<)\Vert_s 
			&\leq C_{P}2^d(\Vert u^>\Vert_{s_0}\Vert u^<\Vert_{s}^{d-2}+\Vert u^>\Vert_{s_0}^2\Vert u^<\Vert_s^{d-3}),
		\end{align*}
		and
		$$\Vert u^>\Vert_{s_0}^2=\sum_{|J|>N}e^{2s_0 f(|J|)}|u_{J}|^2=\sum_{|J|>N}\frac{e^{2sf(|J|)}|u_J|^2}{e^{2(s-s_0)f(|J|)}}\leq\frac{\Vert u\Vert_s^2}{e^{2(s-s_0)f(|N|)}}.$$
		So we get 
		$$\sup_{u\in B^s(r)}\Vert(X_P)(u^>,u^<)\Vert_s\leq C_{P}\frac{2^dr^{d-1}}{e^{(s-s_0)N}}, $$
		which means
		$$|P|_{r,s}\leq C_{P}\frac{2^dr^{d-2}}{e^{(s-s_0)f(N)}}.$$
	\end{proof}

	\begin{lemma}[Lie bracket estimate]\label{Lie}
		Given two polynomials $P\in\mathcal{P}_{p},Q\in\mathcal{P}_{q},|Q|_{r,s}\leq\delta:=\frac{\rho}{8e(r+\rho)},$ we have $\{P,Q\}\in\mathcal{P}_{p+q-2}$ and $|\{P,Q\}|_{r,s}\leq |P|_{r+\rho,s}|Q|_{r+\rho,s}\frac{1}{2\delta}$. Besides, 
		$$|ad_{Q}^kP|_{r,s}\leq|P|_{r+\rho,s}(\frac{|Q|_{r+\rho,s}}{2\delta})^k.$$
	\end{lemma}
	The proof can be seen in Appendix B in \cite{BMP20}.
	
	\begin{lemma}[Estimate for $W$ function]\label{Lam}
		For $x<-2,xe^x=y<0,x=W_{-1}(y)$ satisfies
		$$0<\ln(-\frac{1}{y})<-W_{-1}(y)<2\ln(-\frac{1}{y}).$$
	\end{lemma}
	\begin{proof}
		When $x<-2,-e^{x}>xe^x>-e^\frac{x}{2}$, we have
		\begin{align*}
			-e^{W_{-1}(y)}&>y>-e^{\frac{W_{-1}(y)}{2}},\\
			W_{-1}(y)&<\ln(-y)<\frac{W_{-1}(y)}{2}<0,\\
			0&<-\frac{W_{-1}(y)}{2}<\ln(-\frac{1}{y})<-W_{-1}(y),
		\end{align*}
		as desired.
	\end{proof} 
	\begin{lemma}\label{weight}
		$f=(x+2)^\theta,f=(\ln(x+\kappa))^q$ both satisfy assumption A.0.
	\end{lemma}  
	\begin{proof}
		The first two properties are obvious; here, we only need to prove the third one.
		For $f(x)=x^\theta$, we take $c=2,C_f=2^{\theta-1}$, and prove the case of two variables:
		$$(x_1+x_2)^\theta\leq x_1^\theta+2^{\theta-1}x_2^\theta,x_1\geq2,x_2\geq2.$$ 
		By making the homogenizing substitution $t=\frac{x_2}{x_1}\leq1$, it suffices to prove:
		$$(1+t)^\theta\leq1+\frac{(2t)^\theta}{2},0<t\leq1,$$
		which is easy to calculate.
		Then for $x_d<...<x_1$, we can get
		$$f(\sum_{l=1}^{d}x_l)\leq\sum_{l=1}^{d}C_f^{l-1}f(x_l)\leq f(x_1)+C_f\sum_{l=2}^{d}f(x_l).$$
		
		For $f(x)=(\ln x+\kappa)^q,$ we take $F(x)=f(x+x_2)-f(x)$. Notice that 
		$$f''=q(\ln x+\kappa)^{q-2}(\frac{q-1-\ln(x+\kappa)}{x+\kappa}),$$ so when $\kappa=e^q, f''<0$, $F'(x)=f'(x+x_2)-f'(x)<0$, we thus have $F(x_1)\leq F(x_2)=f(2x_2)-f(x_2)$. Hence we just need to prove:
		$$f(2x_2)-f(x_2)\leq C_f f(x_2),$$
		namely $$\frac{\ln (2x_2+\kappa)}{\ln(x_2+\kappa)}\leq(1+C_f)^{\frac{1}{q}},$$
		which is hold for $x_2$ large enough. Then we can also use induction to come to the conclusion.

	\end{proof}
	
	\begin{lemma}\label{det}
		Let $u^{(1)},...,u^{(k)}$ be $k$ independent vectors with $\Vert u^{(i)} \Vert_{\ell^1}\leq1.$ Let $w\in\mathbb{R}^k$ be an arbitrary vector, then there exists $i\in\{1,...,K\},$ such that 
		$$|u^{(i)}\cdot w|\geq\frac{\Vert w\Vert_{\ell^1}\text{det}(u^{(1)},...,u^{(k)})}{k^{\frac{3}{2}}}.$$
	\end{lemma}
	\begin{lemma}\label{measure}
		Suppose that $g(\tau)$ is $m$ times differentiable on an interval $J\subset\mathbb{R}.$ Let $J_{h}:=\{m\in J\mid |g(\tau)|<h \},h>0.$ If on $J,|g^{(m)}(\tau)|\geq d>0,$ then $|J_h|\leq2(2+3+...+m+d^{-1})h^{\frac{1}{m}}.$ The $|\cdot|$ here is Lebesgue measure. 
	\end{lemma}
	
	Lemma \ref{det} and Lemma \ref{measure} are from Lemma 5.2 and Lemma 5.4 in  \cite{BG06}.

	   \section*{Data Availability \& Competing Interests }
	  Data sharing is not applicable to this article as no datasets were generated or analysed during the current study.
	 
	  We have no conflicts of interest to disclose.

		\section*{Acknowledgement}
		The second author (Y. Li) was supported by National Natural Science Foundation of China (12071175, 12471183 and 12531009).

\end{document}